\documentclass[11pt]{article}
\usepackage[centertags]{amsmath}
\usepackage{amsfonts}
\usepackage{amssymb}
\usepackage{amsthm}
\usepackage{newlfont}
\usepackage[pdftex]{graphicx}
\usepackage{epstopdf}

\newtheorem{Theorem}{Theorem}[section]
\newtheorem{Definition}[Theorem]{Definition}
\newtheorem{Proposition}[Theorem]{Proposition}
\newtheorem{Lemma}[Theorem]{Lemma}

\newtheorem{Remark}[Theorem]{Remark}
\newtheorem{Example}[Theorem]{Example}

\newtheorem{Hypothesis}{Hypothesis}

\numberwithin{equation}{section}

\begin{document}
\renewcommand{\figurename}{Fig.1}

\def\le{\left}
\def\r{\right}
\def\cost{\mbox{const}}
\def\a{\alpha}
\def\d{\delta}
\def\ph{\varphi}
\def\e{\epsilon}
\def\la{\lambda}
\def\si{\sigma}
\def\La{\Lambda}
\def\B{{\cal B}}
\def\A{{\mathcal A}}
\def\L{{\mathcal L}}
\def\O{{\mathcal O}}
\def\bO{\overline{{\mathcal O}}}
\def\F{{\mathcal F}}
\def\K{{\mathcal K}}
\def\H{{\mathcal H}}
\def\D{{\mathcal D}}
\def\C{{\mathcal C}}
\def\M{{\mathcal M}}
\def\N{{\mathcal N}}
\def\G{{\mathcal G}}
\def\T{{\mathcal T}}
\def\R{{\mathcal R}}
\def\I{{\mathcal I}}

\def\bw{\overline{W}}
\def\phin{\|\varphi\|_{0}}
\def\s0t{\sup_{t \in [0,T]}}
\def\lt{\lim_{t\rightarrow 0}}
\def\iot{\int_{0}^{t}}
\def\ioi{\int_0^{+\infty}}
\def\ds{\displaystyle}
\def\pag{\vfill\eject}
\def\fine{\par\vfill\supereject\end}
\def\acapo{\hfill\break}

\def\beq{\begin{equation}}
\def\eeq{\end{equation}}
\def\barr{\begin{array}}
\def\earr{\end{array}}
\def\vs{\vspace{.1mm}   \\}
\def\rd{\reals\,^{d}}
\def\rn{\reals\,^{n}}
\def\rr{\reals\,^{r}}
\def\bD{\overline{{\mathcal D}}}
\newcommand{\dimo}{\hfill \break {\bf Proof - }}
\newcommand{\nat}{\mathbb N}
\newcommand{\E}{\mathbb E}
\newcommand{\Pro}{\mathbb P}
\newcommand{\com}{{\scriptstyle \circ}}
\newcommand{\reals}{\mathbb R}

\title{Averaging principle for non autonomous slow-fast systems of stochastic RDEs: the almost periodic case\thanks{{\em Key words}: Averaging principle, stochastic reaction-diffusion systems, evolution families of measures, almost periodic functions}}

\author{Sandra Cerrai\thanks{Partially supported by the NSF grant DMS 1407615.}\\
\normalsize University of Maryland, College Park\\ United States
\and
Alessandra Lunardi\thanks{Partially supported by the PRIN 2010/11 project 2010MXMAJR}\\
\normalsize Universit\`a di Parma\\ Italy}
\date{}

\maketitle

\begin{abstract} 
We study the validity of an averaging principle for a slow-fast system of stochastic reaction diffusion equations. We assume here that the coefficients of the fast equation depend on time, so that the classical formulation of the averaging principle in terms of the invariant measure of the fast equation is not anymore available. As an alternative, we introduce the time depending evolution family of measures associated with the fast equation. Under the assumption that the coefficients in the fast equation are almost periodic, the evolution family of measures is almost periodic. This allows to identify the appropriate averaged equation and prove the validity of the averaging limit.  
 
\end{abstract}

\section{Introduction}
\label{sec1}
We deal with a class of
systems of stochastic partial differential equations of
reaction-diffusion type on  a bounded domain $D$ of
$\mathbb{R}^d$, with $d\geq 1$,
\begin{equation}
\label{eq0} \le\{
\begin{array}{l}
\ds{\frac{\partial u_\e}{\partial t}(t,\xi)=\mathcal{A}_1
u_\e(t,\xi)+b_1(\xi,u_\e(t,\xi),v_\e(t,\xi))+g_1(\xi,u_\e(t,\xi))\,\frac{\partial
w^{Q_1}}{\partial t}(t,\xi),}\\
\vs 
\ds{\frac{\partial v_\e}{\partial t}(t,\xi)=\frac
1\e\le[(\mathcal{A}_2(t/\e)-\a)
v_\e(t,\xi)+b_2(t/\e,\xi,u_\e(t,\xi),v_\e(t,\xi))\r] }\\
\vs\ds{\ \ \ \ \ \ \ \ \ \ \ \ \ \ \ +\frac 1{\sqrt{\e}}\,
g_2(t/\e,\xi,v_\e(t,\xi))\,\frac{\partial
w^{Q_2}}{\partial t}(t,\xi),}\\
\vs \ds{u_\e(0,\xi)=x(\xi),\ \ \ \ v_\e(0,\xi)=y(\xi),\ \ \ \ \
\xi
\in\,D,}\\
\vs \ds{ \mathcal{N}_{1} u_\e\,(t,\xi)=\mathcal{N}_{2}
v_\e\,(t,\xi)=0,\ \ \ \ t\geq 0,\ \ \ \ \xi \in\,\partial D,}
\end{array}\r.
\end{equation}
where $\e$ is  a small positive parameter and $\a$ is  fixed  positive constant. The operator $\mathcal{A}_2$ and the functions $b_2$ and $g_2$ in the fast equation are  allowed to depend on time. We assume that $\mathcal{A}_2$ is periodic, and $b_2$, $g_2$ are almost periodic in time.

In a series of papers (\cite{av2}, \cite{pol} and \cite{cf}),    the validity of an averaging principle  for some classes of slow-fast stochastic reaction-diffusion systems has been investigated, in the case the fast equation coefficients do not depend on time. It has been proved that the slow motion $u_\e$ converges in $C([0,T];L^2(D))$, as $\e\downarrow 0$, to the solution $\bar{u}$ of the so-called averaged equation, obtained by taking the average of the coefficients $b_1$ and $g_1$ (in case they both depend on the fast motion) with respect to the invariant measure of the fast motion, with frozen slow component (see next formulas \eqref{aveintro} and \eqref{bintro}). Moreover, in \cite{normal} the fluctuations of $u_\e$ around the averaged motion $\bar{u}$ have been studied. More precisely, it has been proven that  the normalized difference $z_\e:=(u_\e-\bar{u})/\sqrt{\e}$ is weakly convergent in $C([0,T];L^2(D))$, as $\e\downarrow 0$, to a process $z$, which is given in terms of a Gaussian process whose covariance is explicitly described. Other aspects of the averaging principle for slow-fast systems of SPDEs have been  studied by several other authors, see e.g. \cite{drw}, \cite{fd}, \cite{kukpia}, \cite{massei} and \cite{rw}.

\medskip

Unlike in all the above mentioned papers, where only the time-independent case has  been considered, in the present paper we deal with non-autonomous systems of reaction-diffusion equations of Hodgkin-Huxley or Ginzburg -Landau type, perturbed by a Gaussian noise of multiplicative type.
Such systems arise in many areas in biology and physics and have attracted considerable attention. In particular, in  neurophysiology  the Hodgkin-Huxley model, and its simplified version given by the Fitzhugh-Nagumo   system, are used to describe the activation and deactivation dynamics of a spiking neuron (see e.g. \cite{tuck1} for a mathematical introduction to this theory). The classical Hodgkin-Huxley model has time-independent coefficients, but, as mentioned by Wainrib in \cite{wainrib}, where  an analogous problem for finite dimensional systems has been studied, systems with time-dependent coefficients are particularly important to study models of learning in neuronal activity and, for this reason  are  worth of a thorough analysis.

Such analysis does not follow in a straightforward manner from   results already available in the mathematical literature. On the contrary, it requires the introduction of some new ideas and techniques. 

Actually, in the standard setting of time independent coefficients, the averaged motion $\bar{u}$ solves the equation
\begin{equation}
\label{aveintro}
\le\{\begin{array}{l}
\ds{\frac{\partial \bar{u}}{\partial t}(t,\xi)=\mathcal{A}_1
\bar{u}(t,\xi)+\bar{B}(\bar{u}(t))(\xi)+
g_1(\xi,\bar{u}(t,\xi))\,\frac{\partial
w^{Q_1}}{\partial t}(t,\xi), }\\
\vs \ds{\bar{u}(0,\xi)=x(\xi),\ \ \ \ \
\xi
\in\,D,\ \ \ \ \ \ \mathcal{N}_{1}
\bar{u}(t,\xi)=0,\ \ \ \ t\geq 0,\ \ \ \ \xi \in\,\partial D.}
\end{array}\r.
\end{equation}
In the equation above, the averaged non-local coefficient $\bar{B}$ is defined by
\begin{equation}
\label{bintro}
\bar{B}(x)=\int_{ C(\bar{D})}B_1(x,z)\,\mu^x(dz),\ \ \ \ \ x \in\,C(\bar{D}),
\end{equation}
where  $\mu^x$ is the invariant measure of the fast equation with frozen slow component $x  \in\,C(\bar{D})$
\begin{equation}
\label{fastintro}
\le\{\begin{array}{l}
\ds{\frac{\partial v^{x,y}}{\partial t}(t,\xi)=(\mathcal{A}_2-\a)
v^{x,y}(t,\xi)+b_2(\xi,x(\xi),v^{x,y}(t,\xi))+
g_2(\xi,x(\xi),v^{x,y}(t,\xi))\,\frac{\partial
w^{Q_2}}{\partial t}(t,\xi), }\\
\vs \ds{v^{x,y}(s,\xi)=y(\xi),\ \ \ \ \
\xi
\in\,D,\ \ \ \ \ \ \mathcal{N}_{2}
v^{x,y}(t,\xi)=0,\ \ \ \ t\geq 0,\ \ \ \ \xi \in\,\partial D,}
\end{array}\r.
\end{equation}
and  $B_1(x,z)(\xi)=b_1(\xi,x(\xi),z(\xi))$, for any $x, z \in\,C(\bar{D})$ and $\xi \in\,\bar{D}$.

Furthermore, as proven in \cite{pol}, 
\begin{equation}
\label{nl125-intro}
\E\,\le|\frac 1T\int_t^{t+T}B_1(x,v^{x,y}(s))\,ds-\bar{B}(x)\r|_{C(\bar{D})}\leq \a(T)\le(1+|x|_{C(\bar{D})}^{\kappa_1}+|y|_{C(\bar{D})}^{\kappa_2}\r),
\end{equation}
for some  function $\a:[0,\infty)\to [0,\infty)$ such that 
\[\lim_{T\to \infty }\a(T)=0.\]

Here, as $\mathcal{A}_2$, $b_2$ and $g_2$ depend on time, we do not have anymore an invariant measure for the fast equation with frozen slow component. Nevertheless, we can prove that there exists an {\em evolution system of probability measures} $\{\mu^x_yt\,;\,t \in\,\mathbb{R}\}$ on $C(\bar{D})$ for equation \eqref{fastintro}, indexed by $t \in\,\reals$. This means that $\mu^x_t$ is a probability measure on $C(\bar{D})$, for any $t \in\,\reals$, and, if $P^x_{s,t}$ is the transition evolution operator associated with equation \eqref{fastintro},
it holds
\[
\int_{C(\bar{D})} P^x_{s,t}\varphi(y)\,\mu^x_s(dy)=\int_{C(\bar{D})} \varphi(y)\,\mu^x_t(dy),\ \ \ \ \ s<t,
\]
for every $\varphi \in\,C_b(C(\bar{D}))$. Moreover, we  show that, under suitable dissipativity conditions, 
\begin{equation}
\label{gap-intro}
\le|P_{s,t}^x\varphi(y)-\int_{C(\bar{D})}\varphi(z)\,\mu^x_t(dz)\r|\leq \|\varphi\|_{C^1_b(C(\bar{D}))}\,e^{-\d (t-s)}\le(1+|x|_{C(\bar{D})}+|y|_{C(\bar{D})}\r),\end{equation}
for some positive constant $\d>0$.

Now, the next fundamental step consists in identifying an  averaged motion $\bar{u}$ as the solution of a suitable averaged equation. Unfortunately, due to the lack of an invariant measure, we do not have anything like \eqref{bintro}.
Still, due to the assumption that $\mathcal{A}_2$ is periodic and both $b_2$ and $g_2$ are almost periodic in time, and to the fact that for any fixed $R>0$ the family of measures 
\[\La_R:=\le\{\mu^x_t\,;\, t \in\,\mathbb{R},\ x \in\,B_R(C(\bar{D}))\r\}\]
is tight in $\mathcal{P}(C(\bar{D}))$, by proceeding as in \cite{dpt} we can prove that the mapping
\[t \in\,\mathbb{R}\mapsto \mu^x_t \in\,\mathcal{P}(C(\bar{D})),\]
is almost periodic, for every $x \in\,C(\bar{D})$. 

This allows us to  find an alternative way to define $\bar{B}$. 
Actually, we  prove that 
the family of functions 
\[
\le\{t \in\,\reals \mapsto   \int_EB_1(x,z)\,\mu^x_t(dz) \in\,C(\bar{D})\,:\,x \in\,K\r\}
\]
is uniformly almost periodic. Then, because of almost periodicity,  we can define
\begin{equation}
\label{bbar}
\bar{B}(x):=\lim_{T\to\infty}\frac 1T\int_0^T\int_{C(\bar{D})}B_{1}(x,y)\,\mu_t^{x}(dy)\,dt,\ \ \ \ x \in\,C(\bar{D}).\end{equation}
Of course, in order to prove that  equation \eqref{aveintro}, with $\bar{B}$ defined as in \eqref{bbar}, is well posed in $C([0,T];C(\bar{D}))$, we need that $\bar{B}$ satisfies some nice properties. Since $B_1$ is not Lipschitz continuous,  there is no hope that $\bar{B}$ is Lipschitz continuous. Nonetheless, we show that, as a consequence of the monotonicity of $B_1$ and of some nice properties we have for the evolution family of measures  $\{\mu^x_t\}_{t \in\,\mathbb{R}}$, the mapping $\bar{B}:C(\bar{D})\to C(\bar{D})$ is locally Lipschitz continuous and some monotonicity holds. And this guarantees the well posedness of equation \eqref{aveintro}.

Next, in the same spirit of \eqref{nl125-intro}, by using  \eqref{gap-intro} we  show that
\[
\E\,\le|\frac 1T\int_s^{s+T}B_1(x,v^{x}(t;s,y))\,dt-\bar{B}(x)\r|^2_{C(\bar{D})}\leq\frac cT\le(1+|x|_{C(\bar{D})}^{\kappa_1}+|y|_{C(\bar{D})}^{\kappa_2}\r)+\a(T,x),\]
for some mapping $\a:[0,\infty)\times C(\bar{D})\to [0,+\infty)$ such that 
\[\lim_{T\to \infty}\,\a(T,x)=0.\]
This allows us to adapt to the present situation the classical Khasminskii method, based on localization in time,  and to prove the main result of this paper, namely that for any fixed $\eta>0$
\begin{equation}
\label{limav-intro}
\lim_{\e\to 0}\ \Pro\le(\sup_{t \in\,[0,T]}|u_\e(t)-\bar{u}(t)|_{C(\bar{D})}>\eta\r)=0,
\end{equation}
where $\bar{u}$ is the solution of the averaged equation \eqref{aveintro}, with $\bar{B}$ defined as in \eqref{bbar}.

Notice that here, due to the polynomial growth of the coefficients, we have also to proceed with a localization in space, which requires, among other things, a suitable approximation for the family of measures $\{\mu^x_t\}_{t \in\,\mathbb{R}}$.

\section{Notations, hypotheses and a few preliminary results}
\label{sec2}

Let $D$ be a  bounded domain of $\reals^d$, with $d\geq 1$, having  smooth boundary.
Throughout the paper, we shall denote by $H$  the separable  Hilbert space
$L^2(D)$, endowed with the scalar product
\[\le<x,y\r>_H=\int_Dx(\xi)y(\xi)\,d\xi,\]
 and with the corresponding norm
$|\cdot|_H$. We shall denote by ${\cal H}$  the product space $H\times H$, endowed with the scalar product
\[\le<x,y\r>_{{\cal H}}=\int_D\le<x(\xi),y(\xi)\r>_{\reals^2}\,d\xi=\le<x_1,y_1\r>_H+\le<x_2,y_2\r>_H\]
and the corresponding norm $|\cdot|_{{\cal H}}$.

Next, we shall denote by $E$ the Banach space $C(\bar{D})$,
endowed with the sup-norm 
\[|x|_E=\sup_{\xi \in\,\bar{D}}|x(\xi)|,\]
 and the duality $\le<\cdot,\cdot\r>_E$. The product space $E\times E$ will be endowed with the norm
 \[|x|_{E\times E}=\le(\,|x_1|_E^2+|x_2|_E^2\r)^{\frac 12},\]
 and the corresponding duality $\le<\cdot,\cdot\r>_{E\times E}$.
 Finally, for any $\theta \in\,(0,1)$ we shall denote by $C^\theta(\bar{D})$ the subspace of $\theta$-H\"{o}lder continuous functions, endowed with the usual norm
 \[|x|_{C^\theta(\bar{D})}=|x|_E+[x]_\theta=|x|_E+\sup_{\substack{\xi, \eta \in\,\bar{D}\\\xi\neq \eta}}\frac{|x(\xi)-x(\eta)|}{|\xi-\eta|^\theta}.\]

 For any $p \in\,[1,\infty]$, with $p\neq 2$, the norms in $L^p(D)$ and $L^p(D)\times L^p(D)$ will be both denoted by $|\cdot|_p$. If $\d>0$ and $p<\infty$, we will denote by $|\cdot|_{\d,p}$ the norm in $W^{\d,p}(D)$
 \begin{equation}
 \label{new94}
 |x|_{\d,p}:=|x|_p+\le(\int_D\int_D\frac{|x(\xi)-x(\eta)|^p}{|\xi-\eta|^{\d p+d}}\,d\xi\,d\eta\r)^{\frac 1p}.
 \end{equation}

Now, we introduce some
notations which we will use in what follows (for all details we refer e.g. to \cite[Appendix D]{dpz1} and also to \cite[Appendix A]{tesi}). For any $x \in\,E$, we denote
\[M_x=\le\{\,\xi \in\,\bar{D}\,:\, |x(\xi)|=|x|_E\,\r\}.\]
Moreover, for any $x \in\,E\setminus \{0\}$, we set
\[\mathcal{M}_x=\le\{\,\d_{x,\xi} \in\,E^\star\,;\,\xi \in\,M_x\,\r\},\]
where
\[\langle \d_{x,\xi},y\rangle_E=\frac{x(\xi)y(\xi)}{|x|_E},\ \ \ \ \ y \in\,E,\]
and, for $x=0$, we set
\[\mathcal{M}_0=\le\{\,h \in\,E^\star\,:\ |h|_{E^\star}=1\,\r\}.\]
Clearly, we have
\[\mathcal{M}_x\subseteq \partial |x|_E:=\le\{\,h
\in\,E^\star\,;\,|h|_{E^\star}=1,\
\le<h,x\r>_E=|x|_E\,\r\},\]
for every $x \in\,E$, and, due to the characterization $\partial |x|_E$, it is possible to show that if $\# M_x=1$, then $\mathcal{M}_x= \partial |x|_E$. In particular,    
if $u:[0,T]\to E$ is any differentiable mapping, then
\begin{equation}
\label{dersub}
\frac d{dt}^- |u(t)|_E\leq \le<u^\prime(t),\d\r>_E,
\end{equation}
for any $t \in\,[0,T]$ and $\d \in\,\mathcal{M}_{u(t)}.$

Analogously, if $x \in\,E\times E$, we set 
\[M_x=\le\{\,\xi=(\xi_1,\xi_2) \in\,\bar{D}\times \bar{D}\,:\,|x_1(\xi_1)|=|x_1|_E,\ |x_2(\xi_2)|=|x_2|_E\,\r\}.\]
Moreover, for $x \in\,E\times E\setminus\{0\}$, we set
\[\mathcal{M}_x=\le\{\,\d_{x,\xi} \in\,(E\times E)^\star\,;\,\xi \in\,M_x\,\r\},\]
where
\[\langle \d_{x,\xi},y\rangle_{E\times E}=\frac{x_1(\xi_1)y_1(\xi_1)+x_2(\xi_2)y_2(\xi_2)}{|x|_{E\times E}},\]
and, for $x=0$, we set
\[\mathcal{M}_0=\le\{\,h \in\,(E\times E)^\star\,:\ |h|_{(E\times E)^\star}=1\,\r\}.\]
As above, we have 
\[\mathcal{M}_x\subseteq \partial\,|x|_{E\times E}:=\le\{\,h
\in\,(E\times E)^\star\,;\,|h|_{(E\times E)^\star}=1,\
\le<h,x\r>_{E\times E}=|x|_E\,\r\},\]
and \eqref{dersub} holds true, with $E$ replaced by $E\times E$.

\medskip

Now, let $X$ be any Banach space. We shall denote by $B_b(X)$ the space of bounded Borel
functions $\varphi:X\to \reals$. $B_b(X)$ is a Banach space, endowed with the sup-norm
\[\|\varphi\|_\infty:=\sup_{x \in\,X}|\varphi(x)|.\]
$UC_b(X)$ will be  the subspace of uniformly continuous mappings.
Moreover, we shall denote by $\mathcal{L}(X)$ the space of bounded
linear operators on $X$ and, in the case $X$ is a Hilbert space, we shall denote by $\mathcal{L}_2(X)$  the subspace of
Hilbert-Schmidt operators, endowed with the norm
\[\|Q\|_{\mathcal{L}_2(X)}=\sqrt{\text{Tr}\,[Q^\star Q]}.\]

\bigskip

The stochastic perturbations in the slow and in the fast
 motion
equations \eqref{eq0} are given respectively by the Gaussian
noises $\partial w^{Q_1}/\partial t(t,\xi)$ and $\partial
w^{Q_2}/\partial t(t,\xi)$, for $t\geq 0$ and $\xi \in\,D$, which
are assumed to be white in time and colored in space, in the case
of space dimension $d>1$. Formally, the cylindrical Wiener
processes $w^{Q_i}(t,\xi)$ are defined by
\begin{equation}
\label{wiener}
w^{Q_i}(t,\xi)=\sum_{k=1}^\infty Q_i e_{k}(\xi)\,\beta_{k}(t),\ \ \ \ i=1,2,\end{equation}
 where
 $\{e_{k}\}_{k \in\,\nat}$ is a
complete orthonormal basis in $H$, $\{\beta_{k}(t)\}_{k
\in\,\nat}$ is a sequence of mutually independent standard
Brownian motions defined on the same complete stochastic basis
$(\Omega,\mathcal{F}, \mathcal{F}_t, \mathbb{P})$ and $Q_i$ is a
bounded linear operator on $H$.

\subsection{The operators $\mathcal{A}_1$ and $\mathcal{A}_2(t)$}

The operators $\mathcal{A}_1$ and $\mathcal{A}_2(t)$, $t \in\,\reals$,  are second order
uniformly elliptic operators, having continuous coefficients on
$D$, and the boundary operators $\mathcal{N}_1$ and
$\mathcal{N}_2$ can be either the identity operator (Dirichlet
boundary condition) or a first order operator with $C^1$ coefficients satisfying a uniform
nontangentiality condition.

In what follows, we shall assume that the operator $\mathcal{A}_2(t)$  has the following form 
\begin{equation}
\mathcal{A}_2(t)=\gamma(t)  \mathcal{A}_2+\mathcal{L}(t),\ \ \ \ \ \ t \in\,\reals,\end{equation}
where $\mathcal{A}_2$ is a second order
uniformly elliptic operator, having continuous coefficients on
$D$, independent of $t$,  and ${\cal L}(t)$ is a first order differential operator of the following form
\[{\cal L}(t,\xi)u(\xi)=\le<l(t,\xi),\nabla u(\xi)\r>_{\reals^d},\ \ \ \ t \in\,\mathbb{R},\ \ \xi \in\,D.\]

\begin{Hypothesis}
\label{H0}
\begin{enumerate}
\item The function $\gamma:\reals\to\reals$ is continuous and there exist $\gamma_{0},\gamma_{1}>0$ such that
\begin{equation}
\label{gamma}
\gamma_{0}\leq \gamma(t)\leq \gamma_{1},\ \ \ \ t\in\,\reals.\end{equation}
\item The function $l:\reals\times D\to\reals^d$ is continuous and  bounded.
\end{enumerate}

\end{Hypothesis}

The realizations  $A_i$, with $i=1,2$,  of the differential
operators $\mathcal{A}_i$ in the spaces $L^p(D)$ and $C(\overline{D})$, endowed with the domains
\[D(A_i^{(p)})=\le\{ f \in\,W^{2,p}(D)\ :\ \mathcal{N}_i f=0\ \text{at}\ \partial D\r\},\ \ \ \ i=1,2,\]
and
\[D(A_i)=\le\{ f \in\,\bigcap_{q>1} W^{2,q}(D)\ :\ \mathcal{A}_i f \in\,C(\overline{D}),\ \mathcal{N}_i f=0\ \text{at}\ \partial D\r\},\ \ \ \ i=1,2,\]
generate analytic semigroups in $L^p(D)$, $1<p<\infty$, and in $E$, respectively. Since $A_i^{(p)}$ is an extension of $A_i$ and $e^{t A_i^{(p)}}$ is an extension of $e^{t A_i}$, we shall drop the indices and write $A_i$ and $e^{t A_i}$ even working in $X=L^p(D)$.

\medskip

As
in \cite{av2} and \cite{pol}, we assume that the operators $A_1$, $A_2$ and
$Q_1$, $Q_2$  satisfy the following conditions.

\begin{Hypothesis}
\label{H1}

For $i=1, 2$ there exist a complete
orthonormal system $\{e_{i,k}\}_{k \in\,\nat}$ of $H$, which is
contained in $C^1(\overline{D})$, and two sequences of non-negative real
numbers $\{\a_{i,k}\}_{k \in\,\nat}$ and $\{\la_{i,k}\}_{k
\in\,\nat}$  such that
\[A_i\, e_{i,k}=-\a_{i,k}\, e_{i,k},\ \ \ \ \ Q_i e_{i,k}=\la_{i,k} e_{i,k},\ \ \ k \geq 1,\]
and
\[ \kappa_i:=\sum_{k=1}^\infty
\la_{i,k}^{\rho_i}\,|e_{i,k}|_\infty^2<\infty,\ \ \ \ \
\zeta_i:=\sum_{k=1}^\infty
\a_{i,k}^{-\beta_i}\,|e_{i,k}|_\infty^2<\infty,
\]
for some  constants $\rho_i \in\,(2,+\infty]$ and $\beta_i
\in\,(0,+\infty)$   such that
\begin{equation}
\label{new50}
\,\frac{\beta_i(\rho_i-2)}{\rho_i}<1.
\end{equation}

\end{Hypothesis}

For comments and examples concerning these assumptions on the
operators $A_i$ and $Q_i$ and the eigenfunction $e_{i,k}$, we refer to \cite[Remark 2.1]{av2} and \cite{grieser}.

For any $t>0$, $\d \in\,[0,2]$ and $p>1 1$ the semigroups $e^{t A_i}$ map $L^p(D)$ into $W^{\d,p}(D)$ with
\begin{equation}
\label{new8}
|e^{t A_i}x|_{\d,p}\leq c_i\,(t\wedge 1)^{-\frac \d 2}|x|_p,\ \ \ \ x \in\,L^p(D).
\end{equation}
By  the Sobolev Embedding Theorem, this implies that the semigroups $e^{tA_i}$ map $L^p(D)$ into $L^q(D)$, for any $1< p\leq q$, and
\begin{equation}
\label{new9}
|e^{t A_i}x|_{q}\leq c_i\,(t\wedge 1)^{-\frac{d(q-p)}{2pq} }|x|_p,\ \ \ \ x \in\,L^p(D).
\end{equation}
Moreover, $e^{tA_i}$ maps $C(\overline{D})$  into $C^\theta(\bar{D})$, for any $\theta(0,2)$, with
\begin{equation}
\label{new10}
|e^{t A_i}x|_{C^\theta(\bar{D})}\leq c_i\,(t\wedge 1)^{-\frac \theta 2}|x|_E.
\end{equation}

\medskip

Now,  we define
\[\gamma(t,s):=\int_s^t\gamma(r)\,dr,\ \ \ \ s<t,\]
and for any $\e>0$ and $\la\geq 0$  we set
\begin{equation}
\label{fin1001}
U_{\la,\e}(t,s)=e^{\frac 1\e\gamma(r,\rho)A_2-\frac \la\e(t-s)},\ \ \ \ s<t.
\end{equation}
In the case $\e=1$, we write $U_{\la}(t,s)$ and in the case $\e=1$ and $\la=0$ we  write $U(t,s)$.
Next,    for any $\e>0$, $\la\geq 0$ and for any $u \in\,C([s,t];W^{1,p}_0(D))$ and $r \in\,[s,t]$, we define
 \begin{equation}
\label{psilambda}
\psi_{\la,\e}(u;s)(r)=\frac 1\e\int_s^r U_{\la,\e}(r,\rho)L(\rho) u(\rho)\,d\rho,\ \ \ \ s<r<t.
\end{equation}
Moreover, for $\e=1$ we simply write $\psi_{\la}(u;s)(r)$.

\begin{Lemma}
For any $s<t$ the operator $e^{\gamma(t,s)A_2}L(s)$ can be extended as a linear operator both in $L^p(D)$, with $1< p<\infty$, and in $E$. Moreover, for any $\eta>0$, its extension (still denoted by $e^{\gamma(t,s)A_2}$) satisfies
\begin{equation}
\label{stimaL2}
\|e^{\gamma(t,s)A_2}L(s)\|_{\mathcal{L}(E)}\leq c_\eta\,((t-s)\wedge 1)^{-(\frac 12+\eta)}.
\end{equation}
\end{Lemma}

\begin{proof}
Let $f \in\,W^{1,p}_0(D)$. For any $0<s<t$ and $\varphi \in\,L^{p^\prime}(D)$, as $e^{\gamma(t,s)A_2}$ is self-adjoint we have
\[\int_D \le(e^{\gamma(t,s)A_2}L(s)f\r)(x)\varphi(x)\,dx=\int_DL(s)f(x)\,e^{\gamma(t,s)A_2}\varphi(x)\,dx.\]
Therefore, if we integrate by parts, due to \eqref{new8} (with $\d=1$) and \eqref{gamma}, we get
\[\begin{array}{l}
\ds{\le|\int_D \le(e^{\gamma(t,s)A_2}L(s)f\r)(x)\varphi(x)\,dx\r|=\le|\int_D f(x)\,D_i\le(l_i(s,\cdot) e^{\gamma(t,s)A_2}\varphi\r)(x)\,dx\r|}\\
\vs
\ds{\leq c\,\le((t-s)\wedge 1\r)^{-\frac 12}|f|_{L^p(D)}\,|\varphi|_{L^{p^\prime}(D)}.}
\end{array}\]
Due to the arbitrariness of $\varphi \in\,L^{p^\prime}(D)$, this yields
\[\le|e^{\gamma(t,s)A_2}L(s)f\r|_{L^p(D)}\leq c\,\le((t-s)\wedge 1\r)^{-\frac 12}|f|_{L^p(D)}.\]
Due to the density of $W^{1,p}_0(D)$ in $L^p(D)$, the operator$e^{\gamma(t,s)A_2}L(s)$ has a bounded linear extension to  $L^p(D)$ (still denoted by $e^{\gamma(t,s)A_2}$) that satisfies  
\begin{equation}
\label{via}
\|e^{\gamma(t,s)A_2}L(s)\|_{\mathcal{L}(L^p(D))}\leq c\,\le((t-s)\wedge 1\r)^{-\frac 12}.
\end{equation}

Now, we fix $\d \in\,(0,1)$ and $p>d/\d$, so that $W^{\d,p}(D)$ is continuously embedded  in $C(\bar{D})$. For any $0<s<t$, we write
\[e^{\gamma(t,s)A_2}L(s)=e^{\gamma(t,(t-s)/2)A_2}e^{\gamma((t-s)/2,s)A_2}L(s).\]
The operator $e^{\gamma(t,(t-s)/2)A_2}$ maps $L^p(D)$ into $W^{\d,p}(D)$, with
\[\|e^{\gamma(t,(t-s)/2)A_2}\|_{\mathcal{L}(L^p(D),W^{\d,p}(D))}\leq c((t-s)\wedge 1)^{-\frac \d 2}.\]
Using the semigroup law and \eqref{via}, we obtain that  $e^{\gamma(t,s)A_2}L(s)$ maps $L^p(D)$ into $W^{\d,p}(D)$ with
\[\|e^{\gamma(t,s)A_2}L(s)\|_{\mathcal{L}(L^p(D),W^{\d,p}(D))}\leq c\,\le((t-s)\wedge 1\r)^{-\frac {1+\d}2}.\]
Now, as $C(\bar{D})$ is continuously embedded  continuously in any $L^p(D)$ and $W^{\d,p}(D)$ is continuously embedded  in $C(\bar{D})$, for $p>d/\d$, we can conclude.
\end{proof}

As a consequence of \eqref{stimaL2}, if we  proceed as in \cite[pages 176-177]{tesi}, we can show that $\psi_{\la,\e}(\cdot;s)$ is a bounded linear operator in $C([s,t];E)$ and there exists a continuous increasing function $c_\la$, with $c_\la(0)=0$, such that for any $s<t$
\begin{equation}
\label{psilambdaunif}
|\psi_{\la,\e}(u;s)|_{C([s,t];E)}\leq c_\la((t-s)/\e)|u|_{C([s,t];E)}.
\end{equation}
Moreover, if $\la>0$ then $c_\la  \in\,L^\infty([0,+\infty))$ and 
\begin{equation}
\label{c-lambda}
\lim_{\la\to\infty} |c_\la|_{\infty}=0.
\end{equation}

\begin{Lemma}
\label{lemma2.1}
For every $\eta \in\,(0,1)$ and $p\geq 1$, there exists $\bar{k}\geq 1$ such that for every
$k\geq \bar{k}$, $s<t$, $0<\d<\la$ and $u \in\,C([s,t];E)$
\begin{equation}
\label{stimapsi-sob}
e^{\d kr}|\psi_{\la}(u;s)(r)|_{\eta,p}^k\leq c_{k}(\la-\d)\int_s^r e^{-(\la-\d)(r-\rho)}e^{\d k\rho}\,|u(\rho)|_E^k\,d\rho,\ \ \ s<r<t,
\end{equation}
for some continuous decreasing function $c_{k}$ such that
\[\lim_{\gamma\to \infty}c_{k}(\gamma)=0.\]
\end{Lemma}

\begin{proof}
Due to \eqref{new8} and \eqref{stimaL2}, for any $\eta \in\,(0,1)$ and $p\geq 1$,we have
\[\begin{array}{l}
\ds{|\psi_{\la}(u;s)(r)|_{\eta,p}\leq c\int_s^r e^{-\la(r-\rho)}((r-\rho)\wedge 1)^{-\frac {1+\eta}2}|u(\rho)|_{L^p(D)}\,d\rho}\\
\vs
\ds{\leq c\,e^{-\delta r}\int_s^r e^{-(\la-\d)(r-\rho)}((r-\rho)\wedge 1)^{-\frac{1+\e}2}\,e^{\d \rho}\,|u(\rho)|_E\,d\rho.}
\end{array}\]
Therefore, if we take $\bar{k}$ such that $\bar{k}(1+\eta)/2(\bar{k}-1)<1$, for any $k\geq \bar{k}$ we have
\[\begin{array}{l}
\ds{e^{\d kr}|\psi_{\la}(u;s)(r)|_{\eta,p}^k\leq c_k\le(\int_0^{r-s} e^{-(\la-\d)\rho}(\rho\wedge 1)^{-\frac{(1+\eta) k}{2(k-1)}}\,d\rho\r)^{k-1}\int_s^tre^{-(\la-\d)(r-\rho)}e^{\d k\rho}\,|u(\rho)|^k_E\,d\rho.}
\end{array}\]
This implies \eqref{stimapsi-sob}, if we set
\[c_{k}(\gamma)=c_k \le(\int_0^{+\infty} e^{-(\la-\d)\rho}(\rho\wedge 1)^{-\frac{(1+\eta) k}{2(k-1)}}\,d\rho\r)^{k-1}.\]

\end{proof}

Due to the Sobolev embedding theorem, if we pick $\bar{p}$ large enough such that $\eta \bar{p}>d$, we have that for any $k\geq \bar{k}$
\begin{equation}
\label{stimapsi-cteta}
e^{\d kr}|\psi_{\la}(u;s)(r)|_{C^\theta(\bar{D})}^k\leq c_{k}(\la-\d)\int_s^r e^{-(\la-\d)(r-\rho)}e^{\d k \rho}\,|u(\rho)|_E^k\,d\rho,\ \ \ s<r<t,
\end{equation}
where $\theta=\eta-d/\bar{p}$. In particular, for any $k\geq \bar{k}$
 \begin{equation}
\label{stimapsi}
e^{\d k r}|\psi_{\la}(u;s)(r)|_{E}^k\leq c_{k}(\la-\d)\int_s^r e^{-(\la-\d)(r-\rho)}e^{\d k \rho}\,|u(\rho)|_E^k\,d\rho,\ \ \ s<r<t.
\end{equation}

\begin{Lemma}
\label{lemmafinl}
For any $u \in\,L^k(s,t;E)$, with $k\geq 1$, and for any $\e>0$ and $\la\geq 0$, it holds
\[|\psi_{\la,\e}(u;s)|_{L^k(s,t;E)}\leq c_{\la, k}((t-s)/\e)|u|_{L^k(s,t;E)}.\]
Moreover, if $\la>0$, then $c_{\la,k} \in\,L^\infty(0,\infty)$ and
\[\lim_{\la\to\infty}|c_{\la,k}|_\infty=0.\]
\end{Lemma}

\begin{proof}
As in the proof of Lemma \ref{lemma2.1}, for any $\eta \in\,(0,1)$ and $p\geq 1$ we have
\[|\psi_{\la,\e}(u;s)(r)|_{\eta,p}\leq \frac c\e \int_s^r e^{-\frac \la\e(r-\rho)}((r-\rho)/\e\wedge 1)^{-\frac{1+\eta}2}|u(\rho)|_{E}\,d\rho.\]
Therefore, if we pick $\bar{p}$ large enough so that $\eta\bar{p}>d$, for any $k \geq 1$ we have, by the Young inequality,

\[\begin{array}{l}
\ds{\int_s^t|\psi_{\la,\e}(u;s)(r)|_E^k\,dr\leq \frac{c_k}{\e^k}\int_s^t\le(\int_s^re^{-\frac \la\e(r-\rho)}((r-\rho)/\e\wedge 1)^{-\frac{1+\eta}2}|u(\rho)|_{E}\,d\rho\r)^k\,dr}\\
\vs
\ds{\leq \frac{c_k}{\e^k} \int_s^t|u(r)|_E^k\,dr\le(\int_0^{t-s} e^{-\frac \la\e r}(r/\e\wedge 1)^{-\frac{1+\eta}2}\,dr\r)^k.}
\end{array}\]
Since
\[\frac 1{\e^k}\le(\int_0^{t-s} e^{-\frac \la\e r}(r/\e\wedge 1)^{-\frac{1+\eta}2}\,dr\r)^k=\le(\int_0^{(t-s)/\e} e^{-\la r}(r\wedge 1)^{-\frac{1+\eta}2}\,dr\r)^k,\]
we conclude by taking
\[c_{\la,k}(\gamma):=\le(\int_0^{\gamma} e^{-\la r}(r\wedge 1)^{-\frac{1+\eta}2}\,dr\r)^k.\]

\end{proof}

\subsection{The coefficients $b_i$ and $g_i$}

As far as the reaction coefficient $b_1:\bar{D}\times \reals^2\to \reals$ 
in the slow equation is concerned, we assume the following condition, that are the same conditions of the paper \cite{pol}.

\begin{Hypothesis}
\label{H2}
\begin{enumerate}
\item The mapping $b_1:\bar{D}\times \reals^2\to \reals$ is continuous and there exists $m_1\geq 1$ such that 
\begin{equation}
\label{boundgb1}
\sup_{\xi \in\,\bar{D}}|b_1(\xi,\si)|\leq
c\,\le(1+|\si_1|^{m_1}+|\si_2|\r),\ \ \ \ \si=(\si_1,\si_2) \in\,\reals^2.
\end{equation}

\item There exists $\theta\geq 0$ such that 
\begin{equation}
 \label{nl1200}
\sup_{\xi \in\,\bar{D}}|b_1(\xi,\si)-b_1(\xi,\rho)|\leq c\,\le(1+|\si|^{\theta}+|\rho|^{\theta}\r)|\si-\rho|,\ \ \ \ \si, \rho \in\,\reals^2.
\end{equation}

\item There exists $c>0$ such that for any $\si, h \in\,\reals^2$ 
\begin{equation}
 \label{nl101}
\sup_{\xi \in\,\bar{D}}\le(b_1(\xi,\si+h)-b_1(\xi,\si)\r)h_1\leq c\,|h_1|\le(1+|\si|+|h|\r).
\end{equation}

\end{enumerate}
\end{Hypothesis}

\begin{Example}
\label{new90} From \cite{pol}.
{\em Let $h: \bar{D}\times \reals\to \reals$ be a continuous function such that $h(\xi,\cdot):\reals\to \reals $ is locally Lipschitz-continuous, uniformly with respect to $\xi \in\,\bar{D}$. Assume that 
\begin{equation}
\label{ll1}
\sup_{\xi \in\,\bar{D}}|h(\xi,s)|\leq c\,\le(1+|s|^m\r),\ \ \ \ \ s \in\,\reals,\end{equation}
and
\begin{equation}
\label{ll2}
h(\xi,s_1)-h(\xi,s_2)=\rho(\xi,s_1,s_2)(s_1-s_2),\ \ \ \ \ \xi \in\,\bar{D},\ \ \ s_1, s_2 \in\,\reals,
\end{equation}
for some $\rho:\bar{D}\times \reals^2\to \reals$ such that
\[\sup_{\substack{\xi \in\,\bar{D}\\s_1, s_2 \in\,\reals}}\rho(\xi,s_1,s_2)<\infty.\]
Moreover, let $k: \bar{D}\times \reals^2\to \reals$ be a continuous function, such that $k(\xi,\cdot):\reals^2\to \reals$ has linear growth and  is locally  Lipschitz-continuous, uniformly with respect to $\xi \in\,\bar{D}$. 

Now, we fix any continuous function $f:\bar{D}\times \reals\to \reals$ such that $f(\xi,\cdot)$ is of class $C^1$, for any $\xi \in\,\bar{D}$, and 
\begin{equation}
\label{cl30}
0\leq \frac{\partial f}{\partial s}(\xi,s)\leq c,\ \ \ \ (\xi,s) \in\, \bar{D}\times \reals,
\end{equation}
for some $c>0$. 
If we define
\[b_1(\xi,\si)=f(\xi,h(\xi,\si_1)+k(\xi,\si_1,\si_2)),\]
it is not difficult to check that conditions 1 and 3 in Hypothesis \ref{H2} are all satisfied. Moreover, if we assume that $h$ and $k$ are differentiable and their derivatives have polynomial growth, then condition 2 is also satisfied.

Next, let $\beta$ and $\beta_i$ be continuous functions from $\bar{D}$ into $\reals$, for $i=1,\ldots,2k$, and assume
\[\inf_{\xi \in\,\bar{D}}\beta(\xi)>0.\]
Then, it is possible to check that the function
\[h(\xi,s):=-\beta(\xi)s^{2k+1}+\sum_{i=1}^{2k}\beta_i(\xi)s^i,\]
satisfies conditions \eqref{ll1} and \eqref{ll2}.
\begin{flushright}
$\Box$

\end{flushright}
}
\end{Example}

For the reaction term $b_2:\reals\times \bar{D}\times \reals^2\to \reals$ in the fast equation, we assume the following conditions.

\begin{Hypothesis}
\label{H2bis}

\begin{enumerate}

\item The mapping $b_2:\reals\times \bar{D}\times \reals^2\to \reals$ is continuous and  there exists $m_2\geq 1$ such that 
\begin{equation}
\label{boundgb2}
\sup_{(t,\xi) \in\,\reals\times \bar{D}}|b_2(t,\xi,\si)|\leq
c\le(1+|\si_1|+|\si_2|^{m_2}\r),\ \ \ \ \si=(\si_1,\si_2) \in\,\reals^2.
\end{equation}

\item The mapping $b_2(t,\xi,\cdot):\reals^2\to\reals$ is locally Lipschitz continuous, uniformly with respect to $(t,\xi)  \in\,\reals\times \bar{D}$. 

\item There exists $c>0$ such that for any $\si, h \in\,\reals^2$ 
\begin{equation}
 \label{nl101bis}
\sup_{(t,\xi) \in\,\reals\times \bar{D}}\le(b_2(t,\xi,\si+h)-b_2(t,\xi,\si)\r)h_2\leq c\,|h_2|\le(1+|\si|+|h|\r).
\end{equation}

\item For every $(t,\xi) \in\,\reals\times \bar{D}$, we have
\begin{equation}
\label{cl41}
b_2(t,\xi,\si_1,\si_2)-b_2(t,\xi,\rho_1,\si_2)=\theta(t,\xi,\si_1,\rho_1,\si_2),\end{equation}
for some
 continuous function $\theta:\reals\times \bar{D}\times \reals^3\to \reals$ such that 
 \begin{equation}
 \label{cl42}
\inf_{\substack{(t,\xi) \in\,\reals\times \bar{D}\\(\si_1,\si_2) \in\,\reals^2,\ h>0}}\theta(t,\xi,\si_1,\si_1+h,\si_2)\sup_{\substack{(t,\xi) \in\,\reals\times \bar{D}\\(\si_1,\si_2) \in\,\reals^2,\ h>0}}\theta(t,\xi,\si_1,\si_1+h,\si_2)\geq 0,
\end{equation}
and such that for any $R>0$ there exists $L_R>0$ with
\begin{equation}
\label{new60}
\si_1, \rho_1 \in\,B_\reals(R)\Longrightarrow \sup_{\substack{(t,\xi) \in\,\reals\times \bar{D}\\\si_2 \in\,\reals}}|\theta(t,\xi,\si_1,\rho_1,\si_2)|\leq L_R\,|\si_1-\rho_1|.
\end{equation}

\item For any $\si_1, \si_2, \rho_2 \in\,\reals$, we have
\begin{equation}
\label{nl102}
b_2(t,\xi,\si_1,\si_2)-b_2(t,\xi,\si_1,\rho_2)=-\la(t,\xi,\si_1,\si_2,\rho_2)(\si_2-\rho_2),
\end{equation}
for some measurable function $\la:\reals\times \bar{D}\times \reals^3\to [0,+\infty)$.

\end{enumerate}
\end{Hypothesis}

\begin{Example}
{\em Let $h:\reals\times \bar{D}\times \reals\to \reals$ be such that $h(t,\cdot)$ satisfies the same conditions as  in Example \ref{new90}, uniformly with respect to $t \in\,\reals$. Assume that the function $\rho$ in \eqref{ll2} depends also on $t \in\,\reals$ and satisfies
\begin{equation}
\label{new91}
\sup_{\substack{(t,\xi) \in\,\reals\times \bar{D}\\s \in\,\reals}}\rho(t,\xi,s)\leq 0.
\end{equation}
Moreover, assume that  the mapping $k:\reals\times \bar{D}\times \reals^2\to \reals$ is continuous, 
the mapping  $k(t,\xi,\cdot):\reals^2\to\reals$ has linear growth and is locally Lipschitz-continuous, uniformly with respect to $(t,\xi) \in\,\reals\times \bar{D}$, and the mapping $k(t,\xi,\cdot,\si_2):\reals\to\reals$ is monotone and locally Lipschitz-continuous, uniformly with respect to $(t,\xi) \in\,\reals\times \bar{D}$ and $\si_2 \in\,\reals$.

Then  all the  conditions  in Hypothesis \ref{H2bis} are  fulfilled if we define 
\[b_2(t,\xi,\si)=f(t,\xi,h(t,\xi,\si_2)+k(t,\xi,\si)),\ \ \ (t,\xi)\in\,\reals\times \bar{D},\ \ \ \si\in\,\reals^2,\]
for any $f:\reals\to\reals$ satisfying \eqref{cl30}.
 Notice that \eqref{new91} holds for
\[h(t,\xi,s)=-\beta(t,\xi)s^{2k+1}+\sum_{j=1}^{2k}\beta_j(t,\xi)s^j-\la s,\] with $\la$ large enough.

}
\end{Example}

Concerning the diffusion coefficients $g_1$ and $g_2$, we assume they satisfy the following conditions.

\begin{Hypothesis}
\label{H3}

\begin{enumerate}
\item The mappings $g_1: \bar{D}\times \reals\to \reals$ and $g_2:\reals\times \bar{D}\times \reals\to \reals$ are continuous and the mappings $g_1(\xi,\cdot):\mathbb{R}\to \mathbb{R}$ and $g_2(t,\xi,\cdot):\bar{D}\times \reals\to \reals$ are
Lipschitz-continuous, uniformly with respect to $\xi \in\, \bar{D}$ and $(t,\xi) \in\,\reals\times \bar{D}$, respectively. 

\item It holds
\begin{equation}
\label{new15}
\sup_{\xi \in\,\bar{D}}|g_1(\xi,\si)|\leq c\le(1+|\si|^{\frac 1{m_1}}\r),\ \ \ \si \in\,\reals,
\end{equation}
and
\begin{equation}
\label{new14}
\sup_{(t,\xi) \in\,\reals\times \bar{D}}|g_2(t,\xi,\si)|\leq c\le(1+|\si|^{\frac{1}{m_2}}\r),\ \ \ \si \in\,\reals,
\end{equation}
\end{enumerate}
where $m_1$ and $m_2$ are the  constants introduced in \eqref{boundgb1} and \eqref{boundgb2}.

\end{Hypothesis}

\begin{Remark}
{\em We are assuming here that the diffusion coefficient $g_2$ in the fast equation does not depend on the slow variable  because of what is required in the proof of Proposition \ref{Prop3}. If the coefficient $b_2$ in the fast equation had linear growth, then we could allow $g_2$ to depend also on the slow variable.}
\end{Remark}

\medskip
In what follows, for any $t \in\,\reals$, and $x, y \in\,E$   we shall set
\[B_1( x,y)(\xi):=b_1(\xi,x(\xi),y(\xi)),\ \ \ \ B_2(t, x,y)(\xi):=b_2(t,\xi,x(\xi),y(\xi)),\ \ \ \ \xi \in\,\bar{D},\]
and
\[B(t):=(B_1,B_2(t)),\ \ \ \ t \in\,\reals.\]

Due to Hypotheses \ref{H2} and \ref{H2bis}, the mappings $B_1$ and $B_2$ are well defined and continuous from $ E\times E$ and $\reals\times E\times E$, respectively, to $E$, so that $B:\reals\times E\times E\to E\times E$ is well defined and continuous. As the mappings $b_1$ and $b_2$ have polynomial growth, $B(t)$ is not well defined in $\mathcal{H}$. 

In view of  \eqref{boundgb1} and \eqref{boundgb2},  for any $x, y \in\,E$ and $t \in\,\reals$, we have
\begin{equation}
\label{5.4} |B_1(x,y)|_E\leq c\,\le(1+|x|_E^{m_1}+|y|_E\r),\ \ \ |B_2(t,x,y)|_E\leq c\,\le(1+|x|_E+|y|_E^{m_2}\r),
\end{equation}
so that
\begin{equation}
\label{5.4prod} |B(t,x,y)|_{E\times E}\leq c\,\le(1+|x|_E^{m_1}+|y|_E^{m_2}\r),\ \ \ x,y
\in\,E,\ \ \ t \in\,\reals.
\end{equation}
As a consequence of \eqref{nl101} and \eqref{nl101bis}, it is immediate to check that for any  $x,y,h, k \in\,E$, for any $t \in\,\reals$ and for any $\d \in\,\mathcal{M}_h$
 \begin{equation}
\label{5.5E} \le<B_1(x+h,y+k)-B_1(x,y),\d\r>_{E}\leq c\,\le(1+|h|_{E}+|k|_E+|x|_{E}+|y|_E\r),\end{equation}
and
\[ \le<B_2(t,x+h,y+k)-B_2(t,x,y),\d\r>_{E}\leq c\,\le(1+|h|_{E}+|k|_E+|x|_{E}+|y|_E\r),\]
so that for any $(x, y), (h,k) \in\,E\times E$, for any   $t \in\,\reals$ and $\d \in\,\mathcal{M}_{(h,k)}$
\begin{equation}
\label{5.5H} \le<B(t,x+h, y+k)-B(t,x,y),\d\r>_{E\times E}
\leq c\,\le(1+|(h,k)|_{E\times E}+|(x,y)|_{E\times E}\r).
\end{equation}
Moreover, 
from \eqref{nl102} we have
\begin{equation}
\label{5.5E2} 
\le<B_2(t,x,y+k)-B_2(t,x,y),\d\r>_E\leq 0,
\end{equation}
for every $\d \in\,\mathcal{M}_k$.
Finally,
in view of \eqref{nl1200} we have
\begin{equation}
 \label{5.5k}
\begin{array}{ll}
\ds{|B_1(x_1,y_1)-B_1(x_2,y_2)|_E}\\
\vs
\ds{\leq c\,\le(1+|(x_1,y_1)|_{E\times E}^\theta+|(x_2,y_2)|_{E\times E}^\theta\r)\le(|x_1-x_2|_{E}+|y_1-y_2|_E\r).}
\end{array}
\end{equation}

\medskip

Next, for any $x,y,z \in\,E$ and $t \in\,\reals$ we define
\[[G_1(x)z](\xi)=g_1(\xi,x(\xi))z(\xi),\ \ \ \ [G_2(t,y)z](\xi):=g_2(t,\xi,y(\xi))z(\xi),\ \ \ \ \xi \in\,\bar{D}.\]
Due to Hypothesis \ref{H3},  the mappings
\[G_1: E\to \mathcal{L}(E)\] and, for any fixed $t \in\,\reals$,
\[G_2(t,\cdot):E \to\mathcal{L}(E)\]
are Lipschitz continuous,  so that the same is true for the mapping $G(t)=(G_1,G_2(t))$ defined on $E\times E$, with values in ${\cal L}(E\times E)$.

\section{Almost periodic functions}
\label{sec3}

We recall here some definitions and results about almost periodic  functions. For all details, we refer to the monographs \cite{besi} and \cite{fink} and the paper \cite{bochner}.

In what follows,  $(X, d_X)$ and $(Y,d_Y)$  denote two complete metric spaces. 
For any bounded function $f:\reals\to Y$ and $\e>0$, we define
\[T(f,\e)=\le\{ \tau \in\,\reals\,:\,d_Y(f(t+\tau),f(t))<\e,\ \text{for all}\ t \in\,\reals\r\}.\]
$T(f,\e)$ is called $\e$-translation set of $f$.

\begin{Definition}
\begin{enumerate}
\item  A continuous function $f:\reals\to Y$ is said to be {\em almost periodic} if, for all $\e>0$ the set $T(f,\e)$ is {\em relatively dense} in $\reals$, that is  there exists a number $l_{\e}>0$ such that  $[a,a+l_\e]\cap T(f,\e)\neq \emptyset$, for every $a \in\,\reals$. The number $l_\e$ is called {\em inclusion length}.

\item Let $F\subset X$ and, for any $x \in\,F$, let $f(\cdot,x):\reals\to Y$ be an almost periodic function. The family of functions $\{f(\cdot,x)\}_{x \in\,F}$ is said {\em uniformly almost periodic} if for any $\e>0$
\[T(F,f,\e):=\bigcap_{x \in\,F}T(f(\cdot,x),\e)\]
is relatively dense in $\reals$ and includes an interval around $0$.
\end{enumerate}

\end{Definition}

In what follows,  if $f:\mathbb{R}\to Y$ or $f:\reals\times X\to Y$ and if $\gamma=\{\gamma_n\}_{n \in\,\nat}$ is a sequence in $\reals$, we shall use the notation $T_\gamma f=g$ to say, respectively, that
\[\lim_{n\to\infty}f(t+\gamma_n)=g(t),\ \ \ \text{in}\ Y,\]
and
\[\lim_{n\to\infty}f(t+\gamma_n,x)=g(t,x),\ \ \ \text{in}\ Y,\]
for any $t \in\,\reals$ and $x \in\,X$.

We recall here some characterization of uniformly almost periodic families of functions. 

\begin{Theorem}
Let $F\subset X$ and let $f(\cdot,x):\reals\to Y$ be a continuous function, for any $x \in\,F$.
The following statements are equivalent.
\begin{enumerate}
\item The family $\{f(\cdot,x)\}_{x \in\,F}$ is uniformly almost periodic.
\item For any sequence $\gamma^\prime=\{\gamma^\prime_n\}_{n \in\,\nat}\subset \reals$ there exists a subsequence $\gamma \subset \gamma^\prime$ and a continuous function $g:\reals\times X\to Y$ such that $T_{\gamma} f=g$, uniformly on $\reals\times F$.

\item For every two sequences $\gamma^\prime$ and $\beta^\prime$ in $\reals$ there exist common  subsequences $\gamma\subset \gamma^\prime$ and $\beta\subset \beta^\prime$ such that $T_{\gamma+\beta}f=T_\gamma T_\beta f$, uniformly on $\reals\times F$.
\end{enumerate}
\end{Theorem}

Notice that if $f:\reals\to X$ is a continuous periodic function, with period $\tau$, then for any sequence $\gamma \subset \reals$ there exists $r_\gamma \in\,[0,\tau]$ such that $T_\gamma f(t)=f(t+r_\gamma)$, uniformly with respect  $t \in\,\reals$. 
In fact, if we denote by $H(f)$ the {\em hull} of $f$, that is the set of functions $\{T_\gamma f\,:\ \gamma=\{\gamma_n\}\subset \reals\}$, we have that $f$ is periodic if and only if $H(f)=\{f(\tau+\cdot)\,:\,\tau \in\,\reals\}$.

\medskip

In the case of a  function $f:\reals\to Y$, we have the following characterization of almost periodicity.

\begin{Theorem}
\label{teo3.3}
A continuous function $f:\reals\to Y$ is almost periodic if and only if for every two sequences $\gamma^\prime$ and $\beta^\prime$ in $\reals$ there exist common  subsequences $\gamma\subset \gamma^\prime$ and $\beta\subset \beta^\prime$ such that $T_{\gamma+\beta}f=T_\gamma T_\beta f$,
pointwise on $\reals$.

\end{Theorem}

\medskip

Finally, in \cite[Theorem I.3.2]{besi}  the following important result for almost periodic functions is shown.
\begin{Theorem}
\label{cl11}
\begin{enumerate}
\item There exists the  {\em mean value} in $Y$ of any almost periodic function $f:\reals\to Y$,   that is
\[\exists \lim_{T\to \infty}\frac 1T\int_0^T f(s)\,ds \in\,Y.\]
Moreover, for every $t \in\,\reals$ 
\[ \lim_{T\to \infty}\frac 1T\int_{t}^{t+T} f(s)\,ds= \lim_{T\to \infty}\frac 1T\int_0^T f(s)\,ds,\]
uniformly with respect to $t \in\,\reals$.

\item If $\{f(\cdot,x)\}_{x \in\,F}$ is a uniformly almost periodic family of functions, with $F\subset X$,  then
\[ \exists \lim_{T\to \infty}\frac 1T\int_{t}^{t+T} f(s,x)\,ds= \lim_{T\to \infty}\frac 1T\int_0^T f(s,x)\,ds,\]
uniformly with respect to $t \in\,\reals$ and $x \in\,F$.
\end{enumerate}

\end{Theorem}

\begin{Remark}
{\em The proof of \cite[Theorem I.3.2]{besi} is given for a single almost periodic function $f$. Nevertheless, it is easy to adapt the arguments used in that proof to the case of uniformly almost periodic families of functions, as stated in the second part of Theorem \ref{cl11}.}
\end{Remark}

\section{The slow-fast system}

With the notations introduced in Section \ref{sec2}, system  
\eqref{eq0} can be rewritten in the following abstract form.

\begin{equation}
\label{astrattoEory}
 \le\{\begin{array}{l} \ds{du_{\e}(t)=\le[A_1
u_{\e}(t)+B_{1}(u_{\e}(t),v_{\e}(t))\r]\,dt+G_1(u_{\e}(t))\,dw^{Q_1}(t),}\\
\vs \ds{dv_{\e}(t)=\frac 1\e \le[(A_2(t/\e)-\a)
v_{\e}(t)+B_2(t/\e,u_{\e}(t),v_{\e}(t))\r]\,dt+\frac 1{\sqrt{\e}}\,
G_2(t/\e,v_{\e}(t))\,dw^{Q_2}(t),}
\end{array}\r.
\end{equation}
with initial conditions $u_{\e}(0)=x \in\,E$ and $v_{\e}(0)=y
\in\,E$. 

In \cite[Theorem 5.3]{cerrai1}, a system analogous to \eqref{astrattoEory} has been studied, in the case of coefficients independent of time. Thanks to Lemma \ref{lemma2.1}, as all estimates satisfied by the coefficients in Hypotheses \ref{H1}, \ref{H2}, \ref{H2bis} and \ref{H3} are uniform with respect to $t \in\,\reals$, the arguments used in the proof of   \cite[Theorem 5.3]{cerrai1} can be adapted to the present situation and it is possible to show that, under Hypotheses \ref{H0}, \ref{H1}, \ref{H2}, \ref{H2bis} and  \ref{H3}, for any $\e>0$ and  $x,y \in\,E$ there exists a unique adapted mild solution to problem \eqref{astrattoEory} in $L^p(\Omega;C_b((s,T];E\times E))$, with $s<T$ and $p\geq 1$. 

This means that there exist two unique adapted processes $u_\e$ and $v_\e$ in $L^p(\Omega;C_b((s,T];E))$ such that 
\[ u_\e(t)=e^{t A_1}x+\int_s^t e^{(t-r)A_1}
B_1(u_\e(r),v_\e(r))\,ds+\int_s^t e^{(t-s)A_1}
G_1(u_\e(r))\,dw^{Q_1}(r),\] and
\[\begin{array}{l}
\ds{v_\e(t)=U_{\a,\e}(t,s)y+\frac 1\e \psi_{\a,\e}(v_\e;s)(t)+\frac 1\e \int_s^t U_{\a,\e}(t,r)B_2(r,u_\e(r),v_\e(r))\,dr}\\
\vs
\ds{+\frac 1{\sqrt{\e}}
\int_s^t U_{\a,\e}(t,r)G_2(r,v_\e(r))\,dw^{Q_2}(r),}
\end{array}\]
where, with the same notations as in Section \ref{sec2}, for every $\e>0$,
\[U_{\a,\e}(t,s)=e^{\frac 1\e\gamma(t,s)A_2-\frac {\a}\e(t-s)},\ \ \ \ s<t\]
and 
\[\psi_{\a,\e}(u;s)(r)=\int_s^r U_{\a,\e}(r,\rho)L(\rho) u(\rho)\,d\rho,\ \ \ \ r \in\,[s,t].\]
Recall that in Section \ref{sec2} we have defined
\[U_\a(t,s):=U_{\a,1}(t,s),\ \ \ \ \ \psi_{\a}(u;s)(r):=\psi_{\a,1}(u;s)(r).\]

Thanks to Lemma \ref{lemmafinl}, we can adapt to the present situation the  arguments used in the proof of \cite[Lemma 3.1]{pol}, and it is possible to show that  for any $p\geq 1$ and $s<T$ there exists a constant
$c_{p,s,T}>0$ such that for any $x, y \in\,E$ and $\e \in\,(0,1]$
\begin{equation} \label{stima2}
\E\,\sup_{t \in\,[s,T]}|u_{\e}(t)|_{E}^p\leq
c_{p,s,T}\,\le(1+|x|^p_E+|y|^p_E\r),\end{equation}
and
\begin{equation}
\label{new2}
\E\int_s^T|v_{\e}(t)|_E^p\,dt\leq
c_{p,s,T}\,\le(1+|x|^p_E+|y|^p_E\r),
\end{equation}
for some constants $c_{s,p,T}$ independent of $\e>0$.

Moreover, as in \cite[Proposition 3.2]{pol}, we can show that  there exists $\bar{\theta}>0$ such that for any $\theta \in\,[0,\bar{\theta})$, $x \in\,C^\theta(\bar{D})$, $y \in\,E$ and  $s<T$ 
\begin{equation}
\label{new6}
\sup_{\e \in\,(0,1]}\,\E\,|u_\e|_{L^\infty(s,T;C^\theta(\bar{\theta}))}\leq c_{s,T}\le(1+|x|_{C^\theta(\bar{D})}+|y|_E\r).
\end{equation}

Finally, by proceeding as in \cite[Proposition 4.4]{av2} (see also \cite[Proposition 3.3]{pol}), we can prove that for any $\theta>0$ there exists $\gamma(\theta)>0$ such that for any $T>0$, $p\geq 2$, $x \in\,C^\theta(\bar{D})$, $y \in\,E$ and $r_1,r_2 \in\,[s,t]$
\begin{equation}
\label{fin1000}
\sup_{\e \in\,(0,1)}\E\,|u_\e(r_1)-u_\e(r_2)|_E^p\leq c_p(T)\le(1+|x|_{C^\theta(\bar{D})}^{p m_1}+|y|_E^p\r)|r_1-r_2|^{\gamma(\theta) p}.
\end{equation}

Due to the Kolmogorov Test and the Ascoli-Arzel\`a Theorem, \eqref{new6} and \eqref{fin1000} imply that  the family $\{\mathcal{L}(u_{\e})\}_{\e \in\,(0,1]}$, given by the laws of the solutions $u_{\e}$, is tight in $C([s,T];E)$, for any $x \in\,C^\theta(\bar{D})$, with $\theta>0$, and for any $y \in\,E$. That is for every $\eta>0$ there exists a compact set $K_\eta\subset C([s,T];E)$ such that $\mathbb{P}\le(u_\e \in\,K_\eta\r)\geq 1-\eta$, for every $\e \in\,(0,1]$.

\section{An evolution family of measures for the fast equation}
\label{sec4}

For any frozen slow component $x \in\,E$, any initial condition $y \in\,E$ and any $s \in\,\reals$, we introduce the problem
\begin{equation}
\label{fast}
dv(t)=\le[(A_2(t)-\a)v(t)+B_2(t,x,v(t))\r]\,dt+G_2(t,v(t))\,d\bar{w}^{Q_2}(t),\ \ \ v(s)=y,
\end{equation}
where $A_2(t)=\gamma(t)A_2+L(t)$ and 
\[\bar{w}^{Q_2}(t)=\le\{
\begin{array}{ll}
\ds{w^{Q_2}_1(t),}  &  \ds{ \text{if}\ t\geq 0,}\\
& \vs
\ds{w^{Q_2}_2(-t), } &  \ds{\text{if}\ t<0,}
\end{array}\r.
\] for two independent $Q_2$-Wiener processes, $w_1^{Q_2}(t)$ and $w_2^{Q_2}(t)$, both defined as in \eqref{wiener}.
A $\{\mathcal{F}_t\}_{t\geq s}$-adapted process $v^x(\cdot;s,y) \in\,L^p(\Omega;C([s,T];E))$ is a {\em mild solution} of \eqref{fast} if
\[\begin{array}{l}
\ds{v^x(t;s,y)=U_{\a}(t,s)y+\psi_{\a}(v^x(\cdot;s,y);s)(t)}\\
\vs
\ds{+\int_s^t U_{\a}(t,r)\,B_2(r,x,v^x(r;s,y))\,dr+\int_s^t U_{\a}(t,r)\,G_2(r,v^x(r;s,y))\,d\bar{w}^{Q_2}(r),}
\end{array}\]
where $\psi_{\a}(\cdot;s)$ is the linear bounded operator defined  in \eqref{psilambda}, with $\e=1$. 

Moreover, if $C(\mathbb{R};E)$ is the space of continuous paths on $\mathbb{R}$ with values in $E$, endowed with the topology of uniform convergence on bounded intervals, a $\{\mathcal{F}_t\}_{t \in\,\reals}$-adapted process $v^x \in\,L^p(\Omega;C(\mathbb{R};E))$ is a {\em mild solution} of the equation 
\begin{equation}
\label{fast-R}
dv(t)=\le[(A_2(t)-\a)v(t)+B_2(t,x,v(t))\r]\,dt+G_2(t,v(t))\,d\bar{w}^{Q_2}(t),
\end{equation}
in $\reals$ if, for every $s<t$,
\[\begin{array}{l}
\ds{v^x(t)=U_{\a}(t,s)v^x(s)+\psi_{\a}(v^x;s)(t)}\\
\vs
\ds{+\int_s^t U_{\a}(t,r)\,B_2(r,x,v^x(r))\,dr+\int_s^t U_{\a}(t,r)\,G_2(r,v^x(r))\,d\bar{w}^{Q_2}(r).}
\end{array}\]

According to \eqref{psilambdaunif}, the mapping $\psi_{\a}(\cdot;s):C([s,T];E)\to C([s,T];E)$ is Lipschitz continuous, so that we can adapt the proof of \cite[Theorem 5.3]{cerrai1} to the present situation and we have that for any $x,y \in\,$ there exists a unique mild solution $v^x(\cdot;s,y) \in L^p(\Omega;C((s,T];E)\cap L^\infty((s,T);E))$, with $p\geq 1$ and $s<T$. 

\medskip

All this allows us to introduce, for any fixed $x \in\,E$,  the transition evolution operator
\[P^x_{s,t}\,\varphi(y)=\mathbb{E}\,\varphi(v^x(t;s,y)),\ \ \ \ s<t,\ \ y \in\,E,\]
where  $\varphi \in\,B_b(E).$

\medskip

For any $\la>0$, equation \eqref{fast} can be rewritten as
\[dv(t)=\le[(A_2(t)-\la)v(t)+B_{2,\la}(t,x,v(t))\r]\,dt+G_2(t,v(t))\,d\bar{w}^{Q_2}(t),\ \ \ \ v(s)=y,\]
where
\[B_{2,\la}(t,x,y)=B_2(t,x,y)+(\la-\a)\,y.\]

In what follows, for any $x \in\,E$ and any process $u \in\,L^p(\Omega;C_b((s,T];E))$, adapted, we shall set
\begin{equation}
\label{new19}
\Gamma_\la(u;s)(t)=\int_s^t U_{\la}(t,r)\,G_2(r,u(s))\,d\bar{w}^{Q_2}(s),\ \ \ \ t>s.
\end{equation}
By proceeding as in the proof of  \cite[Lemma 7.1]{cerrai2}, where the case $s=0$ was considered, it is possible to show that there exists $\bar{p}>1$ such that for any  $p\geq \bar{p}$ and $0<\d<\la$ and for any $u,v \in\,L^p(\Omega;C_b((s,t];E))$, with $s<t$,
\begin{equation}
\label{new15bis}
\begin{array}{l}
\ds{\sup_{r \in\,[s,t]} e^{\d p (r -s)}\E\,|\Gamma_\la(u;s)(r)-\Gamma_\la(v;s)(r)|_E^p}\\
\vs
\ds{\leq c_{p,1}\frac{L_{g_2}^p}{(\la-\d)^{c_{p,2}}}\sup_{r \in\,[s,t]} e^{\d p(r-s)}\,\E\,|u(r)-v(r)|_E^p,}
\end{array}
\end{equation}
where $L_{g_2}$ is the Lipschitz constant of $g_2$ and $c_{p,1}, c_{p,2}$ are two suitable positive constants, independent of $\la>0$ and  $s<t$.

Moreover, using \eqref{new14}, we can show that
\begin{equation}
\label{new16}
\sup_{r \in\,[s,t]} e^{\d p(r-s)}\E\,|\Gamma_\la(u;s)(r)|_E^p\leq c_{p,1}\frac{M_{g_2}^p}{(\la-\d)^{c_{p,2}}}\sup_{r \in\,[s,t]} e^{\d p(r-s)}\,\le(1+\E\,|u(r)|_E^{\frac{p}{m_2}}\r),
\end{equation}
where
\[M_{g_2}=\sup_{\xi \in\,\bar{D},\ \si \in\,\reals}\frac{|g_2(\xi,\si)|}{1+|\si|^{\frac 1{m_2}}},\]
(see \cite[Remark 3.2]{cerrai2}). In fact, in \cite{cerrai2} it is shown that there exists some $\eta>0$ such that for any $p\geq 1$ large enough
\[\sup_{r \in\,[s,t]} e^{\d p(r-s)}\E\,|\Gamma_\la(u;s)(r)|_{\eta,p}^p\leq c_{p,1}\frac{M_{g_2}^p}{(\la-\d)^{c_{p,2}}}\sup_{r \in\,[s,t]} e^{\d p(r-s)}\,\le(1+\E\,|u(r)|_E^{\frac{p}{m_2}}\r).
\]
This means that if we pick $\bar{p}\geq 1$ such that $\eta \bar{p}>d$ and define $\theta=\eta -d/\bar{p}$, by the Sobolev embedding theorem we have that for any $p\geq \bar{p}$
\begin{equation}
\label{new16-cteta}
\sup_{r \in\,[s,t]} e^{\d p(r-s)}\E\,|\Gamma_\la(u;s)(r)|_{C^\theta(\bar{D})}^p\leq c_{p,1}\frac{M_{g_2}^p}{(\la-\d)^{c_{p,2}}}\sup_{r \in\,[s,t]} e^{\d p(r-s)}\,\le(1+\E\,|u(r)|_E^{\frac{p}{m_2}}\r).
\end{equation}

Now, for any fixed adapted process $u \in\,L^p(\Omega;C_b((s,T];E))$, let us introduce the  problem
\begin{equation}
\label{linear}
dz(t)=(A_2(t)-\la)z(t)\,dt+G_2(t,u(t))\,d\bar{w}^{Q_2}(t),\ \ \ \ z(s)=0,
\end{equation}
and let us denote by $\La_\la(u;s)$ its unique mild solution in $L^p(\Omega;C_b((s,T];E))$.
This means that $\La_\la(u;s)$ solves the equation
\[\La_\la(u;s)(t)=\psi_\la(\La_\la(u;s);s)(t)+\Gamma_\la(u;s)(t),\ \ \ s<t<T.\]
Due to Lemma \ref{lemma2.1}, for any $0<\d<\la$ and  $p\geq 1$ large enough, and for any two adapted processes $u_1$ and $ u_2$ in  $L^p(\Omega;C_b((s,T];E))$, with $s<t$, we have
\[\begin{array}{l}
\ds{e^{\d p (t-s)}\,\mathbb{E}\,|\La_\la(u_1;s)(t)-\La_\la(u_2;s)(t)|^p_E\leq c_p\,e^{\d p (t-s)}\, \mathbb{E}\,|\psi_\la(\La_\la(u_1;s)-\La_\la(u_2;s);s)(t)|_E}\\
\vs
\ds{+
c_p\,e^{\d p (t-s)}\,\mathbb{E}\,|\Gamma_\la(u_1;s)(t)-\Gamma_\la(u_2;s)(t)|_E^p}\\
\vs
\ds{\leq c_p(\la-\d)\int_0^{t-s}e^{-(\la-\d)\rho}\,d\rho \sup_{\rho \in\,[s,t]}e^{\d p (\rho-s)}\,\mathbb{E}\,|\La_\la(u_1;s)(\rho)-\La_\la(u_2;s)(\rho)|^p_E}\\
\vs
\ds{+c_p\,e^{\d p (t-s)}\,\mathbb{E}\,|\Gamma_\la(u_1;s)(t)-\Gamma_\la(u_2;s)(t)|_E^p.}
\end{array}\]
Therefore, thanks to \eqref{c-lambda}, we can find $\la(\d)>\d$ large enough such that for any $\la\geq \la(\d)$
\[\begin{array}{l}
\ds{\sup_{\rho \in\,[s,t]}e^{\d p (\rho-s)}\,\mathbb{E}\,|\La_\la(u_1;s)(\rho)-\La_\la(u_2;s)(\rho)|^p_E}\\
\vs
\ds{\leq 
c_p\,\sup_{\rho \in\,[s,t]}\,e^{\d p (\rho-s)}\,\mathbb{E}\,|\Gamma_\la(u_1;s)(t)-\Gamma_\la(u_2;s)(\rho)|_E^p.}
\end{array}\]
Due to \eqref{new15bis}, this yields
\begin{equation}
\label{new15tris}
\begin{array}{l}
\ds{
\sup_{r \in\,[s,t]} e^{\d p(r-s)}\E\,|\La_\la(u_1;s)(r)-\La_\la(u_2;s)(r)|_E^p}\\
\vs
\ds{\leq c_{p,1}\frac{L_{g_2}^p}{(\la-\d)^{c_{p,2}}}\sup_{r \in\,[s,t]} e^{\d p(r-s)}\,\E\,|u_1(r)-u_2(r)|_E^p.}
\end{array}
\end{equation}

In the same way we get that
\begin{equation}
\label{new16bis}
\sup_{r \in\,[s,t]} e^{\d p(r-s)}\E\,|\La_\la(u;s)(r)|_E^p\leq c_{p,1}\frac{M_{g_2}^p}{(\la-\d)^{c_{p,2}}}\sup_{r \in\,[s,t]} e^{\d p(r-s)}\,\le(1+\E\,|u(r)|_E^{\frac{p}{m_2}}\r).
\end{equation}

\begin{Proposition}
\label{prop1}
Assume Hypotheses \ref{H0}, \ref{H1}, \ref{H2bis} and  \ref{H3}. Then, there exists $\d>0$  such that for any $x, y \in\,E$ and $p\geq 1$
\begin{equation}
\label{stima8}
 \E\,|v^{x}(t;s,y)|^p_E\leq c_{p}\,\le(1+e^{-\d p
(t-s)}\,|y|_E^p+|x|_E^{p}\r),\ \ \ \ \ s<t.
\end{equation}
\end{Proposition}

\begin{proof}
We set $z_\la(t):=v^{x}(t;s,y)-\La_\la(t)$, where $\La_\la(t)=\La_\la(v^{x}(\cdot;s,y);s)(t)$ is the solution of  problem \eqref{linear}, with $u=v^{x}(\cdot;s,y)$ and $\la>\a$. Thanks to \eqref{5.5E2}, for every $\d \in\,\mathcal{M}_{z_\la(t)}$  we have
\[\begin{array}{l}
\ds{\frac{d}{dt}^{-}|z_\la(t)|_E\leq \le<(A_2(t)-\la)z_\la(t),\d\r>_E}\\
\vs
\ds{+\le<B_{2,\la}(t,x,z_\la(t)+\La_\la(t))-B_{2,\la}(t,x,\La_\la(t)),\d\r>_E}\\
\vs
\ds{+\le<B_{2,\la}(t,x,\La_\la(t)),\d\r>_E\leq -\a\,|z_\la(t)|_E+c\,\le(1+|x|_E+|\La_\la(t)|_E^{m_2}\r)+(\la-\a)\,|\La_\la(t)|_E}\\
\vs
\ds{\leq -\a\,|z_\la(t)|_E+c\,\le(1+|x|_E+|\La_\la(t)|_E^{m_2}\r)+(\la-\a)^{\frac{m_2}{m_2-1}},}
\end{array}\]
last estimate following from the Young inequality.
By comparison we get
\[|z_\la(t)|_E\leq e^{-\a (t-s)}|y|_E+c\le(1+|x|_E+(\la-\a)^{\frac{m_2}{m_2-1}}\r)+c\int_s^t
e^{-\a (t-r)}|\La_\la(r)|_E^{m_2}\,dr,\]
so that for any $p\geq 1$ 
\[\begin{array}{l}
\ds{|v^{x}(t;s,y)|_E^p\leq c_{p} |\La_\la(t)|_E^p+c_{p}\,e^{-\a p (t-s)}|y|_E^p}\\
\vs
\ds{+c_{p}\le(1+|x|_E^{p}+(\la-\a)^{\frac{p m_2}{m_2-1}}\r)+c_{p}\le(\int_s^t
e^{-\a (t-r)}|\La_\la(r)|_E^{m_2}\,dr\r)^p.}
\end{array}\]
Due to \eqref{new16bis}, this implies that we can proceed as in the proof of \cite[Proposition 4.1]{pol} (where \eqref{new16} with $s=0$ is used), and \eqref{stima8} follows.

\end{proof}

The following Proposition gives a generalization to the case of multiplicative noise of \cite[Lemma 2.2]{dpr}. The fact that the diffusion coefficient is not constant makes the proof of the result considerably more complicated, compared to \cite[Lemma 2.2]{dpr}.

\begin{Proposition} 
\label{prop2}
Under Hypotheses \ref{H0}, \ref{H1}, \ref{H2bis} and  \ref{H3}, if  $\a>0$ is large enough and/or $L_{g_2}$ is small enough, for any $t \in\,\reals$ and $x \in\,E$ there exists $\eta^x(t) \in\,L^p(\Omega;E)$, for all $p\geq 1$, such that 
\begin{equation}
\label{limite}
\lim_{s\to-\infty}\E\,|v^x(t;s,y)-\eta^x(t)|_E^p=0,
\end{equation}
for any $y \in\,E$ and $t \in\,\mathbb{R}$. Moreover, for every $p\geq 1$ there exists some $\d_p>0$ such that
\begin{equation}
\label{stimaeta}
\E\,|v^x(t;s,y)-\eta^x(t)|_E^p\leq c_p\,e^{-\d_p(t-s)}\le(1+|x|_E^p+|y|_E^p\r).
\end{equation}
Finally, $\eta^x$ is a mild solution in $\mathbb{R}$ of equation \eqref{fast-R}.

\end{Proposition}

\begin{proof}
If we fix $h>0$ and define
\[\rho(t)=v^x(t;s,y)-v^x(t;s-h,y),\ \ \ t>s,\]
we have that $\rho(t)$ is the unique mild solution of the problem
\begin{equation}
\label{difference}
\le\{\begin{array}{l}
\ds{d\rho(t)=\le[(A_2(t)-\a)\rho(t)+B_2(t,x,v^x(t;s,y))-B_2(t,x,v^x(t;s-h,y))\r]\,dt}\\
\vs{\ \ \ \ \ \ \ \ \ +\le[G_2(t,v^x(t;s,y))-G_2(t,v^x(t;s-h,y))\r]\,d\bar{w}^{Q_2}(t),}\\
\vs
\ds{\rho(s)=y-v^x(s;s-h,y).}
\end{array}\r.\end{equation}

According to \eqref{nl102}, we have
\[B_2(t,x,v^x(t;s,y))-B_2(t,x,v^x(t;s-h,y))=-J^x(t)\rho(t),\]
where
\[J^x(t,\xi)=\la(t,\xi,x(\xi),v^x(t;s,y)(\xi),v^x(t;s-h,y)(\xi)),\ \ \ \ \xi \in\,D.\]
Therefore, if we define
\[K^x(t,\xi)=\frac{g_2(t,\xi,v^x(t;s,y)(\xi))-g_2(t,\xi,v^x(t;s-h,y)(\xi))}{\rho(t)(\xi)},\ \ \ \xi \in\,D,\]
we can rewrite equation \eqref{difference} as
\begin{equation}
\label{diffbis}
\le\{\begin{array}{l}
\ds{d\rho(t)=\le[(A_2(t)-\a)\rho(t)-J^x(t)\rho(t)\r]\,dt+K^x(t)\rho(t)\,d\bar{w}^{Q_2}(t),}\\
\vs
\ds{\rho(s)=y-v^x(s;s-h,y).}
\end{array}\r.\end{equation}
Notice that, due to \eqref{nl102}, we have 
\begin{equation}
\label{sign}
J^x(t,\xi)\geq 0, \ \ \ (t,\xi) \in\,\reals\times D.
\end{equation}
 Moreover, as $g_2(t,\xi,\cdot)$ is assumed
to be Lipschitz continuous, uniformly with respect to $(t,\xi) \in\,\reals\times \bar{D}$, we have that 
\begin{equation}
\label{cl53}
\sup_{(t,\xi) \in\,\reals\times \bar{D}}|K^x(t,\xi)|=\sup_{(t,\xi) \in\,\reals\times \bar{D}}[g_2(t,\xi,\cdot)]_{\text{\tiny{Lip}}}<\infty.\end{equation}

Now, for any ${\cal F}_s$-measurable $y_s \in\,L^2(\Omega;E)$, we introduce the auxiliary problem
\begin{equation}
\label{difftris}
\le\{\begin{array}{l}
\ds{dz(t)=(A_2(t)-\a)z(t)\,dt+K^x(t)z(t)\,d\bar{w}^{Q_2}(t),}\\
\vs
\ds{z(s)=y_s,}
\end{array}\r.\end{equation}
and we denote by $z(t;s,y_s)$ its solution.
By proceeding as in the proof of \eqref{new15tris}, we have that for any $p$ large enough there exist two constants $c_{p,1}$ and $c_{p,2}$ such that for any $0<\d<\a$
\[\sup_{r \in\,[s,t]} e^{\d p(r-s)}\E\,|z(r;s,y_s)|_E^p\leq c_p\,\E\,|y_s|_E^p+c_{p,1}\frac{L_{g_2}^p}{(\a-\d)^{c_{p,2}}}\sup_{r \in\,[s,t]} e^{\d p(r-s)}\,\E\,|z(r;s,y_s)|_E^p.\]
Therefore, if we pick $\a>0$ large enough and/or $L_{g_2}$ small enough  so that
\[c_{p,1}\frac{L_{g_2}^p}{\a^{c_{p,2}}}<1,\]
we can find $0<\bar{\d}_p<\a$ such that 
\[c_{p,1}\frac{L_{g_2}^p}{(\a-\bar{\d}_p)^{c_{p,2}}}<1.\]
This implies that
\[\sup_{r \in\,[s,t]} e^{p \bar{\d}_p(r-s)}\E\,|z(r;s,y_s)|_E^p\leq c_p\, \E\,|y_s|_E^p,\]
so that 
\begin{equation}
\label{cl1}
\E\,|z(r;s,y_s)|_E^p\leq c_p\,e^{-\d_p (r-s)}\, \E\,|y_s|_E^p,\ \ \ \ s<r,
\end{equation}
with $\d_p=p \bar{\d}_p$.

Next, for any ${\cal F}_s$-measurable $y_s \in\,L^2(\Omega;E)$ we introduce the problem
\begin{equation}
\label{difftris}
\le\{\begin{array}{l}
\ds{dz(t)=\le[(A_2(t)-\a)z(t)-J^x(t)z(t)\r]\,dt+K^x(t)z(t)\,d\bar{w}^{Q_2}(t),}\\
\vs
\ds{z(s)=y_s,}
\end{array}\r.\end{equation}
and we denote by $\hat{z}(t;s,y_s)$ its solution.

Due to the linearity of \eqref{difftris}, by a comparison argument (see \cite{donate}) we have
\[y_s\geq 0,\ \ \ \Pro-\text{a.s.} \Longrightarrow  \hat{z}(t;s,y_s)\geq 0,\ \ \ s<t,\ \ \ \Pro-\text{a.s.}\]
Moreover,  in view of   the sign condition \eqref{sign}, again by a comparison argument (see \cite{donate}) we have
\begin{equation}
\label{new29}
y_s\geq 0,\ \ \ \Pro-\text{a.s.} \Longrightarrow  0\leq\hat{z}(t;s,y_s)\leq z(t;s,y_s),\ \ \ s<t,\ \ \ \Pro-\text{a.s.}
\end{equation}
Thanks to \eqref{cl1}, this allows to conclude
\begin{equation}
\label{cl2}
y_s\geq 0,\ \ \ \Pro-\text{a.s.} \Longrightarrow \E\,|\hat{z}(t;s,y_s)|_E^p\leq c_p\,e^{-\d_p ( t-s)}|y_s|_E^2,\ \ \ \ s<t.\end{equation}
Now, as a consequence of  the linearity of problem \eqref{difftris}, we have
\[\begin{array}{l}
\ds{v^x(t;s,y)-v^x(t;s-h,y)=\hat{z}(t;s,y-v^x(s;s-h,y))}\\
\vs
\ds{=\hat{z}(t;s,y-v^x(s;s-h,y)\wedge y)-\hat{z}(t;s,v^x(s;s-h,y)-v^x(s;s-h,y)\wedge y).}
\end{array}\]
Then, thanks to \eqref{stima8} and \eqref{cl2}, we can conclude that for some $\d_p>0$
\begin{equation}
\label{cl3}
\begin{array}{l}
\ds{\E\,|v^x(t;s,y)-v^x(t;s-h,y)|_E^p\leq c_p\,e^{-\d_p ( t-s)}\E\,|y-v^x(s;s-h,y)|_E^p}\\
\vs
\ds{\leq c_p\,e^{-\d_p ( t-s)}\le(|y|_E^p+e^{-\d_p h}|y|_E^p+|x|_E^p+1\r).}
\end{array}
\end{equation}

Therefore, if we take the limit as $s\to-\infty$, due to the completeness of $L^p(\Omega;E)$, this implies that  for any $t \in\,\reals$ and $x,y \in\,E$ there exists $\eta^x(t) \in\,L^p(\Omega;E)$ such that \eqref{limite} holds. Moreover, if we let $h\to \infty$, we obtain \eqref{stimaeta}. 

Next, in order to prove that $\eta^x(t)$ does not depend on $y \in\,E$, we take $y_1, y_2 \in\,E$ and consider the difference
\[\rho(t)=v^x(t;s,y_1)-v^x(t;s,y_2),\ \ \ t>s.\]
The same arguments, used above for the difference $v^x(t;s,y)-v^x(t;s-h,y)$, can be used here for $\rho(t)$ and we have
\[\E\,|v^x(t;s,y_1)-v^x(t;s,y_2)|_E^p\leq c_p\,e^{- \d_p (t-s)}\,|y_1-y_2|_E^p,\ \ \ \ s<t,\]
so that, by taking the limit above as $s\to-\infty$, we get that the limit $\eta^x(t)$ does not depend on the initial condition $y \in\,E$.

Finally, let us show that $\eta^x$ is the mild solution in $\mathbb{R}$ of equation \eqref{fast-R}.
For any $s<t$ and $h>0$ we have
\[\begin{array}{l}
\ds{v^x(t;s-h,0)=U_{\a}(t,s)v^x(s;s-h,0)+\psi_\a(v^x(\cdot;s-h,0);s)(t)}\\
\vs
\ds{+\int_s^t U_{\a}(t,r)B_2(r,x,v^x(r;s-h,0))\,dr+\int_s^t U_{\a}(t,r)G_2(r,v^x(r;s-h,0))\,d\bar{w}^{Q_2}(r).}
\end{array}\]
Due to \eqref{limite} we can take  the limit as $h$ goes to infinity in both sides and we get for any $s<t$
\begin{equation}
\label{cl15}
\begin{array}{l}
\ds{\eta^x(t)=U_{\a}(t,s)\eta^x(s)+\psi_\a(\eta^x;s)(t)}\\
\vs
\ds{+\int_s^t U_{\a}(t,r)B_2(r,x,\eta^x(r))\,dr+\int_s^t U_{\a}(t,r)G_2(r,\eta^x(r))\,d\bar{w}^{Q_2}.}
\end{array}\end{equation}
This means  that $\eta^x(t)$ is a mild solution in $\reals$ of equation \eqref{fast-R}.

\end{proof}

In what follows, for any $t \in\,\reals$ and $x \in\,E$, we shall denote by $\mu^x_t$ the law of the random variable $\eta^x(t)$. Our purpose here is to show that the family $\{\mu^x_t\}_{t \in\,\reals}$ defines an {\em evolution system of probability measures} on $E$ for equation \eqref{fast}, indexed by $t \in\,\reals$. This means that $\mu^x_t$ is a probability measure on $E$, for any $t \in\,\reals$, and
it holds
\begin{equation}
\label{cl4}
\int_E P^x_{s,t}\varphi(y)\,\mu^x_s(dy)=\int_E \varphi(y)\,\mu^x_t(dy),\ \ \ \ \ s<t,
\end{equation}
for every $\varphi \in\,C_b(E)$.

Notice that, due to \eqref{limite} and \eqref{stima8}, for any $p\geq 1$  we have 
\begin{equation}
\label{cl75}
\sup_{t \in\,\reals}\E\,|\eta^x(t)|_E^p \leq c_p\,\le(1+|x|_E^p\r),\ \ \ \ x \in\,E,\end{equation}
so that
\begin{equation}
\label{cl50}
\sup_{t \in\,\reals}\int_E|y|_E^p\,\mu^x_t(dy)\leq c_p\,\le(1+|x|_E^p\r).
\end{equation}

\begin{Proposition}
\label{evo}
Under Hypotheses \ref{H0}, \ref{H1}, \ref{H2bis} and  \ref{H3}, if  $\a>0$ is large enough and/or $L_{g_2}$ is small enough, 
for any fixed $x \in\,E$ the family of probability measures $\{\mu^x_t\}_{t \in\,\reals}$ introduced  above defines an evolution family  of measure for equation \eqref{fast} such that
\begin{equation}
\label{cl5}
\lim_{s\to-\infty}P^x_{s,t}\varphi(y)=\int_E \varphi(y)\,\mu^x_t(dy),
\end{equation}
for any $\varphi \in\,C_b(E)$. Moreover, if $\varphi \in\,C^1_b(E)$, we have
\begin{equation}
\label{cl60}
\le|P_{s,t}^x\varphi(y)-\int_E\varphi(z)\,\mu^x_t(dz)\r|\leq \|\varphi\|_{C^1_b(E)}\,e^{-\d_1 (t-s)}\le(1+|x|_E+|y|_E\r).
\end{equation}
Finally, if $\{\nu_t^x\}_{t \in\,\reals}$ is another evolution family of measures for \eqref{fast}, such that 
\begin{equation}
\label{cl7}
\sup_{t \in\,\reals}\int_E |y|_E\,\nu^x_t(dy)<\infty,
\end{equation}
then $\nu^x_t=\mu^x_t$, for all $t\in\,\reals$ and $x \in\,E$.
\end{Proposition}

\begin{proof}
According to \eqref{limite}, for any $\varphi \in\,C_b(E)$ and $y \in\,E$ we have
\[\lim_{s\to-\infty} P^x_{s,t}\varphi(y)=\lim_{s\to-\infty}\E\,\varphi(v^x(t;s,y))=\E\varphi(\eta^x(t))=\int_E\varphi(y)\mu^x_t(dy).\]
Therefore, since for any $s<r<t$ we have 
\[P^x_{s,r} P^x_{r,t}\varphi(y)=P^x_{s,t}\varphi (y),\ \ \ \ y \in\,E,\]
by taking the limit above in both sides, as $s\to-\infty$, we obtain
\[\int_E P^x_{r,t}\varphi(y)\,\mu^x_r(dy)=\int_E\varphi(y)\,\mu^x_t(dy),\]
which means
 that $\{\mu^x_t\}_{t \in\,\reals}$ is an evolution family of measures,  satisfying \eqref{cl5}.
 
 In order to prove \eqref{cl60}, we have
 \[\begin{array}{l}
 \ds{\le|P_{s,t}^x\varphi(y)-\int_E\varphi(z)\,\mu^x_t(dz)\r|\leq \E\,\le|\varphi(v^x(t;s,y))-\varphi(\eta^x(t))\r|}\\
 \vs
 \ds{\leq \|\varphi\|_{C^1_b(E)}\E\,|v^x(t;s,y)-\eta^x(t)|_E,}
 \end{array}\]
 so that \eqref{cl60} follows from \eqref{stimaeta}.

Next, let us prove uniqueness.  If we show that for any $\varphi \in\,C^1_b(E)$
\begin{equation}
\label{cl6}
\lim_{s\to-\infty}\int_E P_{s,t}\varphi(y)\,\nu_s^x(dy)=\int_E\varphi(y)\,\mu^x_t(dt),
\end{equation}
then, recalling that $\{\nu^x_t\}_{t \in\,\reals}$ is an evolution family, we have that
for any $\varphi \in\,C^1_b(E)$
\[\int_E\varphi(y)\,\nu^x_t(dy)=\int_E\varphi(y)\,\mu^x_t(dy),\ \ \ \ t \in\,\reals,\]
which implies that $\mu^x_t=\nu^x_t$, for any $t \in\,\reals$ and $x \in\,E$.

In order to prove \eqref{cl6}, we notice that due to \eqref{stimaeta} 
\[\begin{array}{l}
\ds{\le|\int_E P_{s,t}\varphi(y)\,\nu_s^x(dy)-\int_E\varphi(y)\,\mu^x_t(dt)\r|\leq \int_E \E\le|\varphi(v^x(t;s,y))-\varphi(\eta^x(t))\r|\,\nu^x_s(dy)}\\
\vs
\ds{\leq \|\varphi\|_{C^1_b(E)}\int_E \E|v^x(t;s,y)-\eta^x(t)|_E\,\nu^s_s(dy)}\\
\vs
\ds{\leq \|\varphi\|_{C^1_b(E)}e^{- \d(t-s)}\le(1+|x|_E+\int_E |y|_E\,\nu^x_s(dy)\r).}
\end{array}\]
Then, as a consequence of condition \eqref{cl7}, we can conclude that \eqref{cl6} holds and, as we have seen, uniqueness follows.
\end{proof}

 Now, we want to study the dependence of $\eta^x(t)$, and hence of $\mu^x_t$, on the parameter $x \in\,E$.
 
 \begin{Proposition}
 \label{Prop3}
Under Hypotheses \ref{H0}, \ref{H1}, \ref{H2bis} and  \ref{H3}, if  $\a>0$ is large enough and/or $L_{g_2}$ is small enough, we have that for any $R>0$ there exists $c_R>0$ such that 
\begin{equation}
\label{cl76}
x_1, x_2 \in\,B_E(R)\Longrightarrow \sup_{t \in\,\reals }\E\,|\eta^{x_1}(t)-\eta^{x_2}(t)|_E^2\leq c_R\,|x_1-x_2|_E^2 .
\end{equation} 
 
 \end{Proposition}
\begin{proof}
In view of \eqref{limite}, it is sufficient to show that for any $R>0$ there exists $c_R>0$ such that
\begin{equation}
\label{cl40}
x_1, x_2 \in\,B_E(R)\Longrightarrow \sup_{s<t}\E\,|v^{x_1}(t;s,0)-v^{x_2}(t;s,0)|_E^2\leq c_R\,|x_1-x_2|_E^2.
\end{equation}
If we define
\[\rho(t)=v^{x_1}(t;s,0)-v^{x_2}(t;s,0),\ \ \ s<t,\]
we have that $\rho(t)$ is the unique mild solution of the problem
\begin{equation}
\label{difference-bis}
\le\{\begin{array}{l}
\ds{d\rho(t)=\le[(A_2(t)-\a)\rho(t)+B_2(t,x_1,v^{x_1}(t;s,0))-B_2(t,x_2,v^{x_2}(t;s,0))\r]\,dt}\\
\vs{\ \ \ \ \ \ \ \ \ +\le[G_2(t,v^{x^1}(t;s,0))-G_2(t,v^{x_2}(t;s,0))\r]\,d\bar{w}^{Q_2}(t),}\\
\vs
\ds{\rho(s)=0.}
\end{array}\r.\end{equation}

According to \eqref{cl41} and \eqref{nl102}, we have
\[B_2(t,x_1,v^{x_1}(t;s,0))-B_2(t,x_2,v^{x_2}(t;s,0))=-J(t)\rho(t)+I(t),\]
where
\[J(t,\xi)=\la(t,\xi,x_1(\xi),v^{x_1}(t;s,0)(\xi),v^{x_2}(t;s,0)(\xi)),\ \ \ \ \xi \in\,D,\]
and 
\[I(t,\xi)=\theta(t,\xi,x_1(\xi),x_2(\xi),v^{x_2}(t;s,0)(\xi)),\ \ \  \xi \in\,D.\]
Therefore, if we define
\[K(t,\xi)=\frac{g_2(t,\xi,v^{x_1}(t;s,0)(\xi))-g_2(t,\xi,v^{x_2}(t;s,0)(\xi))}{\rho(t)},\ \ \ \xi \in\,D,\]
we can rewrite equation \eqref{difference-bis} as
\begin{equation}
\label{diffbis-bis}
\le\{\begin{array}{l}
\ds{d\rho(t)=\le[(A_2(t)-\a)\rho(t)-J(t)\rho(t)+I(t)\r]\,dt+K(t)\rho(t)\,d\bar{w}^{Q_2}(t),}\\
\vs
\ds{\rho(s)=0.}
\end{array}\r.\end{equation}
Notice that, due to \eqref{nl102}, we have 
\begin{equation}
\label{sign-bis}
J(t,\xi)\geq 0, \ \ \ (t,\xi) \in\,[s,+\infty)\times D.
\end{equation}
Due to \eqref{new60} we have
\begin{equation}
\label{cl44}
x_1, x_2 \in\,B_E(R)\Longrightarrow  \sup_{s<t}|I(t)|_E\leq L_R\,|x_1-x_2|_E.
\end{equation}
 Moreover, due to \eqref{cl42}, we can assume, without any loss of generality, that  
 \begin{equation}
 \label{cl43}
 x_1(\xi)\geq x_2(\xi)\Longrightarrow I(t,\xi)\geq 0,\ \ \ \  (t,\xi) \in\,[s,+\infty)\times \bar{D}.
 \end{equation}
 Finally, as $g_2(t,\xi,\cdot)$ is assumed
to be Lipschitz continuous, uniformly with respect to $(t,\xi) \in\,\reals\times \bar{D}$, we have that 
$K(t)$ satisfies \eqref{cl53}.

Thanks to \eqref{cl43}, by a comparison argument we have
\[x_1\geq x_2\Longrightarrow \rho(t)\geq 0,\ \ \ \Pro-\text{a.s.}, \ \ \ \ s<t.\]
Therefore, again by comparison, due to \eqref{sign-bis} we have
\begin{equation}
\label{cl46}
x_1\geq x_2\Longrightarrow 0\leq \rho(t)\leq \hat{\rho}(t),\ \ \ \Pro-\text{a.s.}, \ \ \ \ s<t,\end{equation}
where $\hat{\rho}(t)$ is the solution of the problem 
\[\le\{\begin{array}{l}
\ds{d\hat{\rho}(t)=\le[(A_2(t)-\a)\hat{\rho}(t)+I(t)\r]\,dt+K(t)\hat{\rho}(t)\,d\bar{w}^{Q_2}(t),}\\
\vs
\ds{\hat{\rho}(s)=0.}
\end{array}\r.\]
This means that
\[\hat{\rho}(t)=\psi_\a(\hat{\rho};s)(t)+\int_s^t U_{\a}(t,r)I(r)\,dr+\int_s^t U_\a(t,r)K(r)\hat{\rho}(r)\,d\bar{w}^{Q_2}(r).\]
As a consequence of \eqref{cl44}, by using the same arguments used in the proof of Proposition \ref{prop1}, we get that if $\a$ is large enough and/or $L_{g_2}$ is small enough
\[x_1\geq x_2,\ \ \ x_1, x_2 \in\,B_E(R)\ \Longrightarrow\   \sup_{s<t}\E\,|\hat{\rho}(t)|_E^2\leq c_R\,|x_1-x_2|_E^2,\]
so that, thanks to \eqref{cl46}, we have
\[x_1\geq x_2,\ \ \ x_1, x_2 \in\,B_E(R)\ \Longrightarrow\   \sup_{s<t}\E\,|v^{x_1}(t;s,0)-v^{x_2}(t;s,0)|_E^2\leq c_R\,|x_1-x_2|_E^2.\]
As in the proof of   Proposition \ref{prop2}, the general estimate \eqref{cl40} follows by noticing that 
 \[\begin{array}{l}
 \ds{|v^{x_1}(t;s,0)-v^{x_2}(t;s,0)|_E^2}\\
 \vs
 \ds{\leq 2\,|v^{x_1}(t;s,0)-v^{x_1\wedge x_2}(t;s,0)|_E^2+2\,|v^{x_1\wedge x_2}(t;s,0)-v^{x_2}(t;s,0)|_E^2.}
 \end{array}\]

\end{proof}

 \section{Almost periodicity of the evolution family of measures}

In what follows, we shall assume the following conditions on the operator $A_2(t)$ and the coefficients $b_2(t,\xi,\si)$ and $g_2(t,\xi,\si)$.

\begin{Hypothesis}
\label{H6}
\begin{enumerate}
\item The functions $\gamma:\reals\to (0,\infty)$ and $l:\reals\times \bar{D}\to \reals^d$ are both periodic, with the same period.
\item The families of functions 
\[{\mathcal B}_R:=\le\{b_2(\cdot,\xi,\si)\,:\,\xi \in\,\bar{D},\ \si \in\,B_{\reals^2}(R)\r\},\ \ \ \ {\mathcal G}_R:=\le\{g_2(\cdot,\xi,\si)\,:\,\xi \in\,\bar{D},\ \si \in\, B_{\reals}(R)\r\}\]
 are both uniformly almost periodic, for any $R>0$
\end{enumerate}
\end{Hypothesis}

\begin{Lemma}
\label{lemma6.1}
Under Hypothesis \ref{H6}, for any $R>0$ the family of functions
 \[\le\{B_2(\cdot,x,y)\,:\,(x,y) \in\,B_{E\times E}(R)\r\},\ \ \ \le\{G_2(\cdot,y)\,:\,y\in\,B_E(R)\r\},\]
 are both uniformly almost periodic.
 \end{Lemma}
 
 \begin{proof} Due to the uniform  almost periodicity of the family ${\mathcal B}_R$, for any $\e>0$ there exists $l_{\e,R}>0$ such that in any interval of $\reals$ of length $l_{\e,R}$ we can find $\tau>0$ such that
\[|b_2(t+\tau,\xi,\si)-b_2(t,\xi,\si)|<\e,\ \ \ \ (t,\xi,\si) \in\,\reals\times \bar{D}\times B_{\reals^2}(R).\]
This implies that
\[\begin{array}{l}
\ds{|B_2(t+\tau,x,y)-B_2(t,x,y)|_E}\\
\vs
\ds{= \sup_{\xi \in\,\bar{D}}|b_2(t+\tau,x(\xi),y(\xi))-b_2(t,x(\xi),y(\xi))|<\e,\ \ (t,x,y) \in\,\reals\times B_{E\times E}(R).}
\end{array}\]
In a completely analogous way, we can show that the family $\{G_2(\cdot,y)\,:\,y \in\,B_E(R)\}$ is uniformly almost periodic.
\end{proof}

Now, for any $\mu,\nu \in\,{\cal P}(E)$, we define
\[d(\mu,\nu)=\sup\,\le\{\le|\int_E f(y)\,(\mu-\nu)(dy)\r|,\ |f|_{\text{\tiny{Lip}}}\leq 1\r\},\]
where
\[|f|_{\text{\tiny{Lip}}}=|f|_E+[f]_{\text{\tiny{Lip}}}=|f|_{E}+\sup_{\xi\neq \eta}\frac{|f(\xi)-f(\eta)|}{|\xi-\eta|}.\]
It is known that the space $({\cal P}(E),d)$ is a complete metric space and the distance $d$ generates the weak topology on ${\cal P}(E)$. 

In \cite{dpt} it is proven that if $A_2(\cdot)$ is periodic,  the family of functions \[\le\{B_2(\cdot,x,y)\,:\,(x,y) \in\,B_{E\times E}(R)\r\},\ \ \ \le\{G_2(\cdot,y)\,:\,y\in\,B_E(R)\r\},\]
 are both uniformly almost periodic, for any $R>0$ and the family of measures $\{\mu^x_t\}_{t \in\,\reals}$ is tight in $\mathcal{P}(E)$, then the mapping $t \in\,\reals\mapsto \mu^x_t \in\,\mathcal{P}(E)$ is almost periodic. The proof in \cite{dpt} is based on Theorem \ref{teo3.3}. Actually, it is proved that for every two sequences $\gamma^\prime$ and $\beta^\prime$ in $\reals$ there exist common  subsequences $\gamma\subset \gamma^\prime$ and $\beta\subset \beta^\prime$ such that $T_{\gamma+\beta}\mu^x_\cdot=T_\gamma T_\beta \mu^x_\cdot$,
pointwise on $\reals$.

Unlike in this paper, in \cite{dpt} it is assumed that the coefficients are Lipschitz continuous and the covariance  $Q_2^2$ of the noise is trace-class. But, all the arguments used in \cite{dpt} can be adapted to the present situation without major difficulties. Therefore, in view of Lemma \ref{lemma6.1}, if we prove that the family of measures $\{\mu^x_t\}_{t \in\,\reals},$
is tight in $\mathcal{P}(E)$, we obtain the following result

\begin{Theorem}
\label{cl10}
Under Hypotheses \ref{H0}, \ref{H1}, \ref{H2bis},  \ref{H3} and \ref{H6}, if  $\a>0$ is large enough and/or $L_{g_2}$ is small enough, we have that the mapping
\[t \in\,\reals\mapsto \mu^x_t \in\, \mathcal{P}(E),\]
is almost periodic, for any fixed $x \in\,E$.
\end{Theorem}

Thus, it only remains to prove tightness.

\begin{Lemma}
Under Hypotheses \ref{H0}, \ref{H1}, \ref{H2bis} and \ref{H3},  if $\a$ is sufficiently large and/or $L_{g_2}$ is sufficiently small, there exists $\theta>0$ such that for any $p\geq 1$ and for any $x \in\,E$
\begin{equation}
\label{cl13}
\sup_{t \in\,\reals}\E\,|\eta^x(t)|^p_{C^{\theta}(\bar{D})}\leq c_p\le(1+|x|_E^p\r).
\end{equation}
In particular, the family of measures 
\[\La_R:=\le\{\mu^x_t\,;\,t \in\,\reals,\ x \in\,B_E(R)\r\},\]
is tight in $\mathcal{P}(E)$, for any $R>0$.
\end{Lemma}

\begin{proof}
Due to \eqref{stima8} and \eqref{stimaeta},  with $y=0$,  we have that for any $p\geq 1$
\begin{equation}
\label{cl20}
\sup_{t \in\,\reals}\E\,|\eta^x(t)|_{E}^p\leq c_p\,(1+|x|_E^p).\end{equation}
Moreover, thanks to  \eqref{new10} and \eqref{cl15}, for every  $t \in\,\reals$ and $\theta>0$
\[\begin{array}{l}
\ds{|\eta^x(t)|_{C^\theta(\bar{D})}\leq c\,|\eta^x(t-1)|_E+\le|\psi_\a(\eta^x;t-1)(t)\r|_{C^\theta(\bar{D})}}\\
\vs
\ds{+\int_{t-1}^t|U_{\a}(t,r) B_2(r,x,\eta^x(r))|_{C^\theta(\bar{D})}\,dr+ \le|\Gamma^x_\a(\eta^x,t-1)(t)\r|_{C^\theta(\bar{D})}.  }
\end{array}
\]
According to \eqref{stimapsi-cteta}, \eqref{new16-cteta} and \eqref{new10}, this implies that for some $\theta>0$ and any $0<\d<\a$ and $p\geq 1$
\[\begin{array}{l}
\ds{e^{\d p}\E\,|\eta^x(t)^p|_{C^\theta(\bar{D})}\leq c_p\,\E|\eta^x(t-1)|^p_E+c_p\sup_{r \in\,[t-1,t]}e^{\d p(r-t+1)}\E\le|\psi_\a(\eta^x;t-1)(r)\r|^p_{C^\theta(\bar{D})}}\\
\vs
\ds{+c_p\int_{t-1}^t|U_{\a}(t,r) B_2(r,x,\eta^x(r))|^p_{C^\theta(\bar{D})}\,dr+ c_p\sup_{r \in\,[t-1,t]}e^{\d p(r-t+1)}\E\le|\Gamma^x_\a(\eta^x;t-1)(r)\r|^p_{C^\theta(\bar{D})} }\\
\vs
\ds{\leq c_p\,\E|\eta^x(t-1)|^p_E+c_p\sup_{r \in\,[t-1,t]}e^{\d p(r-t+1)}\E\le|\eta^x(r)\r|^p_{E}}\\
\vs
\ds{+c_p\le(\int_{t-1}^t (t-r)^{-\frac {\theta } 2}\le(1+|x|_E+\E\,|\eta^x(r)|_E^{m_2 }\r)\,dr\r)^p+c_p\sup_{r \in\,[t-1,t]}e^{\d p(r-t+1)}\le(1+\E\le|\eta^x(r)\r|^{\frac p{m_2}}_{E}\r),}
\end{array}\]
so that
\[\begin{array}{l}
\ds{e^{\d p}\E\,|\eta^x(t)|^p_{C^\theta(\bar{D})}\leq c_p\,\le(\E|\eta^x(t-1)|^p_E+1+|x|_E^p\r)+c_p\sup_{r \in\,[t-1,t]}e^{\d p(r-t+1)}\E\le|\eta^x(r)\r|^p_{E}}\\
\vs
\ds{+c_p\le(\int_{t-1}^t (t-r)^{-\frac {\theta p } {2(p-1)}}\,dr\r)^{p-1}\int_{t-1}^t\E\le|\eta^x(r)\r|^{m_2 p}_{E}\,dr.}
\end{array}
\]
If $p\geq 2$, for any  $\theta<1$ we have  that $\theta p/(p-1)<2$. Then, thanks to \eqref{cl20}, we can conclude that \eqref{cl13} holds true, for any $p\geq 2$. Due to the H\"older inequality, \eqref{cl13} holds for any $p\geq 1$.

\end{proof}

\section{The averaged equation}

For any fixed $x \in\,E$, the mapping
$B_{1}(x,\cdot):E\to E$
is continuous and
\begin{equation}
 \label{nl122}
\le|B_{1}(x,y)\r|_E\leq c\le(1+|x|_E^{m_1}+|y|_E\r).
\end{equation}
$B_1$ is unbounded and only locally Lipschitz continuous, but, as a consequence of Proposition \ref{cl10}, it is still possible to prove the following result.

\begin{Lemma}
\label{lemma4.1}
Under the same hypotheses of Proposition \ref{cl10}, for every compact set $K\subset E$, the family of functions 
\begin{equation}
\label{cl72}
\le\{t \in\,\reals \mapsto   \int_EB_1(x,z)\,\mu^x_t(dz) \in\,E\,:\,x \in\,K\r\}
\end{equation}
is uniformly almost periodic.
\end{Lemma}

\begin{proof}
For every $n \in\,\nat$, we define
\[b_{1,n}(\xi,\si_1,\si_2):=\begin{cases}
\ds{b_{1}(\xi,\si_1,\si_2), } &  \ds{\text{if}\ |\si_2|\leq n,}\\
\\
\ds{b_1(\xi,\si_1 ,\si_2 n/|\si_2|),} &  \ds{\text{if}\ |\si_2|> n.}
\end{cases}\]
and we set
\[B_{1,n}(x,y)(\xi)=b_{1,n}(\xi,x(\xi),y(\xi)),\ \ \ \xi \in\,\bar{D}.\]
Clearly, we have that $B_{1,n}(x,\cdot):E\to E$ is Lipschitz continuous and bounded, for any fixed $x \in\,E$, and $B_{1,n}(x,y)=B_1(x,y)$, if $|y|_E\leq n$. Moreover, for any $R>0$ 
\begin{equation}
\label{cl70}
\sup_{|x|_E\leq R}|B_{1,n}(x,\cdot)|_{\text{\tiny{Lip}}_b(E)}:=c_{n,R}<\infty.
\end{equation}
Now, for any $n \in\,\nat$ we have
\[\int_EB_1(x,z)\,\mu^x_t(dz)=\int_EB_{1,n}(x,z)\,\mu^x_t(dz)+\int_{\{|z|_E>n\}}\le(B_1(x,z)-B_{1,n}(x,z)\r)\,\mu^x_t(dz).\]
According to \eqref{cl50} and \eqref{nl122}, we have
\[\begin{array}{l}
\ds{\sup_{t \in\,\reals}\,\le|\int_{\{|z|_E>n\}}\le(B_1(x,z)-B_{1,n}(x,z)\r)\,\mu^x_t(dz)\r|}\\
\vs
\ds{\leq c\,\sup_{t \in\,\reals}\int_{\{|z|_E>n\}}\le(1+|x|_E^{m_1}+|z|_E\r)\,\mu^x_t(dz)\leq \frac c n\le(1+|x|_E^{m_1+1}\r).}
\end{array}\]
This implies that for any $\e>0$ and $R>0$, we can find $\bar{n}=\bar{n}(\e,R) \in\,\nat$, such that
\[\sup_{\substack{ x \in\,B_E(R)\\t \in\,\reals}}\,\le|\int_{\{|z|_E>\bar{n}\}}\le(B_1(x,z)-B_{1,\bar{n}}(x,z)\r)\,\mu^x_t(dz)\r|\leq \frac \e 4,\]
so that for any $t, \tau \in\,\reals$ and $x \in\,B_E(R)$
\[\begin{array}{l}
\ds{\le|\int_EB_1(x,z)\,\mu^x_{t+\tau}(dz)-\int_EB_1(x,z)\,\mu^x_t(dz)\r|}\\
\vs
\ds{\leq \le|\int_EB_{1,\bar{n}}(x,z)\,\mu^x_{t+\tau}(dz)-\int_EB_{1,\bar{n}}(x,z)\,\mu^x_t(dz)\r|+\frac \e 2.}
\end{array}\]

Now, let us define
\[f(t,x)=\int_EB_{1,\bar{n}}(x,z)\,\mu^x_{t}(dz),\ \ \ \ (t,x) \in\,\reals\times E.\]
If we show that, for any compact set $K\subset E$, the family $\{f(\cdot,x)\,:\,x \in\,K\}$ is uniformly almost periodic, we have concluded our proof.

Since, for any $t, \tau \in\,\reals$, we have
\[|f(t+\tau,x)-f(t,x)|_E\leq |B_{1,\bar{n}}(x,\cdot)|_{\tiny{\text{Lip}}_b(E)}\,d(\mu^x_{t+\tau}, \mu^x_t),\]
in view of Theorem \ref{cl10} and \eqref{cl70}, the function $f(\cdot,x)$ is almost periodic, for any $x \in\,E$. 
Moreover, $f$ is continuous in $x \in\,K$, uniformly with respect to $t \in\,\reals$.  
Actually, thanks to \eqref{5.5k}, we have
\[\begin{array}{l}
\ds{|f(t,x)-f(t,y)|_E\leq \E\,|B_{1,\bar{n}}(x,\eta^x(t))-B_{1,\bar{n}}(y,\eta^y(t))|_E}\\
\vs
\ds{\leq c\,\E\le(1+|x|^\theta_E+|y|_E^\theta+|\eta^x(t)|_E^\theta+|\eta^y(t)|_E^\theta\r)\le(|x-y|_E+|\eta^x(t)-\eta^y(t)|_E\r).}
\end{array}\]
Now, as $K$ is compact it is bounded, so that there exists $R>0$ such that $K\subset B_E(R)$. Therefore, due to Proposition \ref{Prop3} and \eqref{cl75}, we can conclude that for any $x, y \in\,K$
\[\sup_{t \in\,\reals} |f(t,x)-f(t,y)|_E\leq c_R\le(|x-y|_H+ \sup_{t \in\,\reals}\le(\E\,|\eta^x(t)-\eta^y(t)|_E^2\r)^{\frac 12}\r)\leq c_R\,|x-y|_E,\]
and this implies  that the family of functions $\{f(t,\cdot)\,:\,t \in\,\reals\}$ is equicontinuous. In \cite[Theorem 2.10]{fink} it is proven that this implies the uniform almost periodicity of the family $\{f(\cdot,x)\,:\,x \in\,K\}$.

\end{proof}

Due to the almost periodicity of the family of mappings \eqref{cl72}, according to Theorem \ref{cl11} we can define
\[\bar{B}(x):=\lim_{T\to\infty}\frac 1T\int_0^T\int_EB_{1}(x,y)\,\mu_t^{x}(dy)\,dt,\ \ \ \ x \in\,E.\]
Thanks to \eqref{cl50} and \eqref{nl122}, we have that
\begin{equation}
\label{nl1230}
|\bar{B}(x)|_E\leq c\le(1+|x|^{m_1}_E\r).
\end{equation}

Actually, in view of \eqref{nl122} we have
\[\le|\frac 1T\int_0^T\int_E B_{1}(x,y)\,\mu_t^{x}(dy)\,dt\r|_E\leq c\,\frac 1T\int_0^T\int_E\le(1+|x|_E^{m_1}+|y|_E\r)\,\mu_t^{x}(dy)\,dt\]
and then, thanks to \eqref{cl50}, we have
\[|\bar{B}(x)|_E\leq c\,\le(1+|x|_E^{m_1}\r)+c_1\,\le(1+|x|_E\r),\]
which implies \eqref{nl1230}.

\medskip

As a consequence of \eqref{stimaeta}, we have the following crucial result.

\begin{Lemma}
\label{teonl}
Under Hypotheses \ref{H0} to \ref{H6},  if $\a$ is sufficiently large and/or $L_{g_2}$ is sufficiently small, there exist some constants $\kappa_1, \kappa_2\geq 0$ such that for any $T>0$, $s \in\,\reals$ and $x, y \in\,E$ 
\begin{equation}
\label{nl125}
\begin{array}{l}
\ds{\E\,\le|\frac 1T\int_s^{s+T}B_1(x,v^{x}(t;s,y))\,dt-\bar{B}(x)\r|^2_E\leq\frac cT\le(1+|x|_E^{\kappa_1}+|y|_E^{\kappa_2}\r)+\a(T,x),}
\end{array}
\end{equation}
for some  mapping $\a:[0,\infty)\times E\to [0,+\infty)$ such that 
\[\sup_{T>0}\,\a(T,x)\leq c\le(1+|x|_E^{m_1}\r),\ \ \ \ x \in\,E,\]
and for any compact set $K\subset E$
\[\lim_{T\to \infty}\,\sup_{x \in\,K}\a(T,x)=0.\]
\end{Lemma}

\begin{proof}
For any fixed  $\Lambda \in\,E^\star$ and $x\in\,E$, we denote by $\Pi^x_\La B_1$ the mapping
\[(t,y) \in\reals\times \,E\mapsto \Pi^x_\La B_1(t,y):=\le<B_1(x,y),\La\r>_E-\int_E\le<B_1(x,z),\La\r>_E\,\mu^x_t(dz) \in\,\reals.\] 
By proceeding as in the proof of \cite[Lemma 2.3]{av2} and \cite[Lemma 5.1]{pol}, we have
\begin{equation}
\label{new33}
\begin{array}{l}
\ds{\E\,\le(\frac 1T\int_s^{s+T}\le[\le<B_1(x,v^x(t;s,y)),\La\r>_E-\int_E\le<B_1(x,z),\La\r>_E\,\mu^x_t(dz)\r]\,dt\r)^2}\\
\vs
\ds{=\frac 2{T^2}\int_s^{s+T}\int_r^{s+T}  \E\le[ \Pi^x_\La B_1(r,v^x(r;s,y))\,P^x_{r,t}\Pi^x_\La B_1(r,v^x(r;s,y))\r]dt\,dr}\\
\vs
\ds{\leq \frac 2{T^2}\int_s^{s+T}\int_r^{s+T} \le(\E\,| \Pi^x_\La B_1(r,v^x(r;s,y))|^2\r)^{\frac 12}\,\le(\E\,|P^x_{r,t}\Pi^x_\La B_1(r,v^x(r;s,y))|^2\r)^{\frac 12}dt\,dr.}
\end{array}
\end{equation}
Due to \eqref{5.4}, \eqref{stima8} and \eqref{cl50}, we have
\begin{equation}
\label{new34}
\begin{array}{l}
\ds{\E\,| \Pi^x_\La B_1(r,v^x(r;s,y))|^2\leq c\le(1+|x|^{2 m_1}_E+\E\,|v^x(r;s,y)|_E^2\r)|\La|^2_{E^\star}}\\
\vs
\ds{
\leq c\le(1+|x|^{2m_1}_E+e^{-2\d (r-s)}|y|_E^2\r)|\La|_{E^\star}^2.}
\end{array}\end{equation}
Moreover,  due to 
\eqref{5.5k}
we have
\[\le|\le<B_1(x,y),\La\r>_E-\le<B_1(x,z),\La\r>_E\r|\leq c\,|y-z|_E\le(1+|x|_E^\theta+|y|_E^\theta+|z|_E^\theta\r)|\La|_{E^\star},\]
so that, thanks to  \eqref{stimaeta} we have
\[\begin{array}{l}
\ds{\E\,|P^x_{r,t}\Pi^x_\La B_1(r,v^x(r;s,y))|^2\leq c\le(1+|x|_E^{2(\theta\vee1)}+|y|_E^{ 2(\theta\vee 1)}\r)|\La|_{E^\star}^2e^{-2 \d(t-r)}.}

\end{array}\]
Therefore, if we plug the estimate above and estimate \eqref{new34} into \eqref{new33}, we get
\begin{equation}
\label{cl73}
\begin{array}{l}
\ds{\E\,\le|\frac 1T\int_s^{s+T}\le[B_1(x,v^x(t;s,y))-\int_E B_1(x,z)\,\mu^x_t(dz)\r]\,dt\r|_E^2}\\
\vs
\ds{\leq c\,\le(1+|x|^{m_1}_E+|y|_E\r)\le(1+|x|_E^{\theta\vee 1}+|y|_E^{ \theta\vee 1}\r)\frac 1{T^2}\int_s^{s+T}\int_r^{s+T}e^{-\d(t-r)}\,dt\,dr}\\
\vs
\ds{\leq c\,\le(1+|x|^{m_1}_E+|y|_E\r)\le(1+|x|_E^{\theta\vee 1}+|y|_E^{ \theta\vee 1}\r)\,\frac 1{T}.}
\end{array}\end{equation}

Next, thanks to Lemma \ref{lemma4.1} and Theorem \ref{cl11}, 
we have that the limit
\[ \lim_{T\to \infty}\frac 1T\int_s^{s+T}\int_E B_1(x,z)\,\mu^x_t(dz)\,dt \in\,E,\]
converges  to $\bar{B}_1(x)$, uniformly with respect to $s \in\,\reals$ and $x$ in any compact  set $K\subset E$. Therefore, if we define
\[\a(T,x)=2\, \le|\frac 1T\int_s^{s+T}\int_E B_1(x,z)\,\mu^x_t(dz)\,dt -\bar{B}(x)\r|^2_E,\]
we can conclude.

\end{proof}

\medskip

\begin{Lemma}
\label{lemma5.2}
Under Hypotheses \ref{H0} to \ref{H6},  if $\a$ is sufficiently large and/or $L_{g_2}$ is sufficiently small,  we have that the mapping
$\bar{B}:E\to E$ is locally Lipschitz-continuous. Moreover, for any $x, h \in\,E$ and $\d \in\,\mathcal{M}_h$
\begin{equation}
\label{nl1202}
 \le<\bar{B}(x+h)-\bar{B}(x),\d\r>_E\leq
c\,\le(1+|h|_E+|x|_E\r).
\end{equation}
\end{Lemma}

\begin{proof}
For any $x_1, x_2 \in\,E$ we have
\[\bar{B}(x_1)-\bar{B}(x_2)=\lim_{T\to \infty} \frac 1 T\int_0^T\E\,\le(B_1(x_1,\eta^{x_1}(t))-B_1(x_2,\eta^{x_2}(t))\r)\,dt,\ \ \ \ \text{in}\ E.\]
By using \eqref{5.5k} we have
\[\begin{array}{l}
\ds{\le|B_1(x_1,\eta^{x_1}(t))-B_1(x_2,\eta^{x_2}(t))\r|_E}\\
\vs
\ds{\leq c\le(1+|x_1|_E^\theta+|x_2|_E^\theta+|\eta^{x_1}(t)|_E^\theta+|\eta^{x_2}(t)|_E^\theta\r)\le(|x_1-x_2|_E+|\eta^{x_1}(t)-\eta^{x_2}(t)|_E\r),}
\end{array}\]
and then, due to \eqref{cl75}, we get
\[\begin{array}{l}
\ds{\sup_{t \in\,\reals}\le|\E\le(B_1(x_1,\eta^{x_1}(t))-B_1(x_2,\eta^{x_2}(t))\r)\r|_E}\\
\vs
\ds{\leq c\le(1+|x_1|_E^\theta+|x_2|_E^\theta\r)
\le(|x_1-x_2|_E+\sup_{t \in\,\reals}\le(\E|\eta^{x_1}(t)-\eta^{x_2}(t)|^2_E\r)^{\frac 12}\r).}
\end{array}\]
Thanks to \eqref{cl76}, this implies that for any  $R>0$
\[x_1, x_2 \in\,B_E(R)\Longrightarrow |B_1(x_1)-B_1(x_2)|_E\leq c_R\,|x_1-x_2|_E.\]

Concerning \eqref{nl1202}, if $\d \in\,\mathcal{M}_h$ we have
\[\le<\bar{B}(x+h)-\bar{B}(x),\d\r>_E=\lim_{T\to \infty}\frac 1T\int_0^T\E\,\le<B_1(x+h,\eta^{x+h}(s))-B_1(x,\eta^{x}(s)),\d\r>_E\,ds.\]
Now, due to \eqref{5.5E} we have
\[\begin{array}{l}
\ds{\le<B_1(x+h,\eta^{x+h}(s))-B_1(x,\eta^{x}(s)),\d\r>_E}\\
\vs
\ds{\leq c\le(1+|x|_E+|h|_E+|\eta^{x+h}(s)|_E+|\eta^{x}(s)|_E\r),}
\end{array}\]
and then, thanks again to \eqref{cl75}, we conclude
\[\begin{array}{l}
\ds{\le<\bar{B}(x+h)-\bar{B}(x),\d\r>_E}\\
\vs
\ds{\leq \limsup_{T\to \infty}\frac 1T\int_0^Tc\le(1+|x|_E+|h|_E+\E|\eta^{x+h}(s)|_E+\E|\eta^{x}(s)|_E\r)\,ds\leq c\le(1+|x|_E+|h|_E\r).}
\end{array}\]
\end{proof}

\medskip

Now, we can introduce the averaged equation
\begin{equation} \label{avepol}
du(t)=\le[A_1
u(t)+\bar{B}(u(t))\r]\,dt+G(u(t))\,dw^{Q_1}(t),\
\ \ \ u(0)=x \in\,E.\end{equation}
In view of Lemma \ref{lemma5.2} and of \cite[Theorem 5.3]{cerrai1}, for any $x \in\,E$, $T>0$ and $p\geq 1$ equation \eqref{avepol} admits a unique mild solution
$\bar{u} \in\,L^p(\Omega;C_b((0,T];E))$. In the next section we will show that the slow motion $u_\e$ converges in probability to the averaged motion $\bar{u}$.

\section{The averaging limit}

In this last section we prove that the slow motion $u_\e$ converges to the averaged motion $\bar{u}$, as $\e\to 0$. The proof of this averaging result is under many respects similar to the proof of \cite[Theorem 61]{pol}.\begin{Theorem}
\label{zebra}
Assume that Hypotheses \ref{H0} to \ref{H6} hold and fix $x \in\,C^\theta(\bar{D})$, for some $\theta>0$, and $y \in\,E$. Then, if $\a$ is large enough and/or $L_{g_2}$ is small enough, for any $T>0$ and $\eta>0$ we have
\begin{equation}
\label{new61}
\lim_{\e\to 0}\ \Pro\le(\sup_{t \in\,[0,T]}|u_\e(t)-\bar{u}(t)|_E>\eta\r)=0,
\end{equation}
where $\bar{u}$ is the solution of the averaged equation \eqref{avepol}.
\end{Theorem}

For any $h \in\,D(A_1)$, the slow motion $u_\e$ satisfies the identity
\[\begin{array}{l}
\ds{\int_Du_\e(t,\xi)h(\xi)\,d\xi=\int_D x(\xi) h(\xi)\,d\xi+\int_0^t\int_D u_\e(s,\xi)A_1h(\xi)\,d\xi\,ds}\\
\vs
\ds{+\int_0^t\int_D \bar{B}(u_\e(s,\cdot))(\xi)h(\xi)\,d\xi\,ds+\int_0^t\int_D[G_1(u_\e(s)h](\xi)dw^{Q_2}(s,\xi)+R_\e(t),}
\end{array}\]
where
\[R_\e(t)=\int_0^t\int_D \le(B_1(u_\e(s),v_\e(s))(\xi)-\bar{B}(u_\e(s))(\xi)\r)h(\xi)\,d\xi\,ds.\]
Therefore, as in \cite{av2} and \cite{pol}, due to the tightness of the family  $\{{\cal L}(u_\e)\}_{\e \in\,(0,1]}$  in ${\cal P}(C([0,T];E))$, in order to prove Theorem \ref{zebra} it is sufficient to prove that
\begin{Lemma}
\label{new62}
Under the same hypotheses of Theorem \ref{zebra}, for any $T>0$ we have
\begin{equation}
\label{re}
\lim_{\e\to 0}\, \E\sup_{t \in\,[0,T]}|R_\e(t)|_E=0.
\end{equation}
\end{Lemma}

\subsection{Proof of Lemma \ref{new62}}

For any $n \in\,\nat$, we define
\[b_{1,n}(\xi,\si_1,\si_2):=\begin{cases}
\ds{b_{1}(\xi,\si_1,\si_2), } &  \ds{\text{if}\ |\si_1|\leq n,}\\
\\
\ds{b_1(\xi,\si_1 n/|\si_1|,\si_2),} &  \ds{\text{if}\ |\si_1|> n,}
\end{cases},\]
and
\[b_{2,n}(t,\xi,\si_1,\si_2):=\begin{cases}
\ds{b_{2}(t,\xi,\si_1,\si_2), } &  \ds{\text{if}\ |\si_1|\leq n,}\\
\\
\ds{b_2(t,\xi,\si_1 n/|\si_1|,\si_2),} &  \ds{\text{if}\ |\si_1|> n,}
\end{cases},\]

In correspondence toForcorresponding composition operator. We have 
\begin{equation}
\label{new66}
x \in\,B_E(n)\Longrightarrow B_{1,n}(x,y)=B_1(x,y),\ \ \ B_{2,n}(t,x,y)=B_2(t,x,y),
\end{equation}
for every $t \in\,\reals$ and $y \in\,E$.
Notice that the mappings $b_{1,n}$  and $b_{2,n}$ satisfy all conditions in Hypotheses \ref{H2} and \ref{H2bis}, respectively.
For any fixed $t \in\,\reals$, $\xi \in\,\bar{D}$ and $\si_2 \in\,\reals$, the mappings $b_{1,n}(\xi,\cdot,\si_2)$ and $b_{2,n}(t,\xi,\cdot,\si_2)$ are Lipschitz continuous and,  in view of \eqref{new60},
\begin{equation}
\label{new65}
\sup_{\substack{(t,\xi) \in\,\reals\times\bar{D}\\\si_2 \in\,\reals}}|b_{2,n}(t,\xi,\si_1,\si_2)-b_{2,n}(t,\xi,\rho_1,\si_2)|\leq c_n\,|\si_1-\rho_1|,\ \ \ \ \si_1, \rho_1 \in\,\reals.
\end{equation}

Moreover, for any $n \in\,\nat$ we define
\[g_{1,n}(\xi,\si_1):=\begin{cases}
\ds{g_{1}(\xi,\si_1), } &  \ds{\text{if}\ |\si_1|\leq n,}\\
\\
\ds{g_1(\xi,\si_1 n/|\si_1|),} &  \ds{\text{if}\ |\si_1|> n.}
\end{cases}\]
The corresponding composition/multiplication operator is denoted by $G_{1,n}$.

Now, for any $n \in\,\nat$ we introduce the system
\begin{equation}
\label{astrattoEoryn}
 \le\{\begin{array}{l} \ds{du(t)=\le[A_1
u(t)+B_{1,n}(u(t),v(t))\r]\,dt+G_{1,n}(u(t))\,dw^{Q_1}(t),}\\
\vs \ds{dv(t)=\frac 1\e \le[(A_2(t/\e)-\a)
v(t)+B_{2,n}(t/\e,u(t),v(t))\r]\,dt+\frac 1{\sqrt{\e}}\,
G_{2}(t/\e,v(t))\,dw^{Q_2}(t),}
\end{array}\r.
\end{equation}
with initial conditions $u(s)=x$ and $v(s)=y$.
We denote by $z_{\e,n}=(u_{\e,n},v_{\e,n})$  its solution.

\medskip
Next, for any $n \in\,\nat$  we introduce the problem
\begin{equation}
\label{fastn} dv(t)=\le[(A_2(t)-\a)
v(t)+B_{2,n}(t,x,v(t))\r]\,dt+G_{2}(t,v(t))\,dw^{Q_2}(t),\ \ \ \ v(s)=y,
\end{equation}
whose solution will be denoted by $v^{x}_n(t;s,y)$.
Thanks to \eqref{new66}, for any $t\geq 0$ we have
\begin{equation}
\label{new68}
v^{x}_n(t;s,y)=\begin{cases}
\ds{v^{x}(t;s,y), } & \ds{ \text{if}\ |x|_E\leq n,}\\
\\
\ds{v^{x_n}(t;s,y),}  &  \ds{\text{if}\ |x|_E>n,}
\end{cases}\end{equation}
where
\[x_n(\xi):=\begin{cases}
\ds{x(\xi)}  &  \ds{\text{if}\ |x(\xi)|<n}\\
\\
\ds{n\,\text{sign}\,x(\xi),}  &  \ds{\text{if}\ |x(\xi)|\geq n.}
\end{cases}\]
This implies that for each $n \in\,\nat$ and $x \in\,E$ there exists an evolution of measures family $\{\mu^{x,n}_t\}_{t \in\,\reals}$ for equation \eqref{fastn} and $\mu^{x,n}_t$ is given by
\[\mu^{x,n}_t=\begin{cases}
\ds{\mu^x_t, } & \ds{ \text{if}\ |x|_E\leq n,}\\
\\
\ds{\mu^{x_n}_t,}  &  \ds{\text{if}\ |x|_E>n.}
\end{cases}.\]
Moreover, due to \eqref{stima8}, for any  $p\geq 1$  we have
\begin{equation}
\label{new70}
\E\,|v^x_n(t;s,y)|_E^p\leq c_{p,n}\le(1+e^{-\d p (t-s)}|y|_E^p\r),\ \ \ \ \ t>s.
\end{equation}

As all coefficients in equation \eqref{fastn} satisfy the same conditions fulfilled by the coefficients of equation \eqref{fast-R}, we have that a result analogous to Lemma \ref{teonl} holds. More precisely, if we define
\[\bar{B}_{n}(x)=\lim_{T\to \infty}\frac 1T\int_0^T\int_E B_{1,n}(x,y)\,\mu^{x,n}_t(dy)\,dt,\]
we have that
\begin{equation}
\label{new68}
\E\,\le|\frac 1T\int_s^{s+T}B_{1,n}(x,v_n^{x}(t;s,y))\,ds-\bar{B}_n(x)\r|_E^2\leq \frac cT\le(1+|x|_E^{\kappa_1}+|y|_E^{\kappa_2}\r)+\a(T,x),
\end{equation}
for some mapping $\a:(0,+\infty)\times E\to [0,+\infty)$ such that 
\begin{equation}
\label{fin1011}
\sup_{T>0}\a(T,x)\leq c(1+|x|_E^{m_1}),\ \ \ \ x \in\,E,
\end{equation}
and
\begin{equation}
\label{fin1010}
\lim_{T\to\infty}\,\sup_{x \in\,K}\a(T,x)=0,
\end{equation}
for every compact set $K\subset E$.
Notice that 
\[|x|_E\leq n\Longrightarrow \bar{B}_n(x)=\bar{B}(x).\]

\begin{Lemma}
The mapping $\bar{B}_{n}:E\to E$ is Lipschitz-continuous.
\end{Lemma}

\begin{proof}
Due to \eqref{nl1200}, for every $t \in\,\reals$ and $x_1, x_2 \in\,E$ we have
\[\begin{array}{l}
\ds{|B_{1,n}(x_1,\eta^{x_1}_n(t))-B_{1,n}(x_2,\eta^{x_2}_n(t))|_E}\\
\vs
\ds{\leq c_n\,|x_1-x_2|_E+|B_{1,n}(x_2,\eta^{x_1}_n(t))-B_{1,n}(x_2,\eta^{x_2}_n(t))|_E}\\
\vs
\ds{\leq c_c\,|x_1-x_2|_E+c_n \le(1+|\eta^{x_1}_n(t)|_E^\theta+|\eta^{x_2}_n(t)|_E^\theta\r)|\eta^{x_1}_n(t)-\eta^{x_2}_n(t)|_E.}
\end{array}\]
Due to \eqref{new70}, we have
\begin{equation}
\label{cl300}
\sup_{t \in\,\reals} \E\,|\eta_n^x(t)|_E^p=:c_{p,n}<\infty,
\end{equation}
and this implies 
\[\begin{array}{l}
\ds{\le|\int_E B_{1,n}(x_1,y)\,\mu_t^{n,x_1}(dy)-\int_E B_{1,n}(x_2,y)\,\mu_t^{n,x_2}(dy)\r|_E}\\
\vs
\ds{\leq \E\,\le|B_{1,n}(x_1,\eta^{x_1}_n(t))-B_{1,n}(x_2,\eta^{x_2}_n(t))\r|_E\leq c_n\,|x_1-x_2|_E+c_n\le(\E\le|\eta^{x_1}_n(t)-\eta^{x_2}_n(t)\r|_E^2\r)^{\frac 12}.}
\end{array}\]
Now, if we adapt the proof of Proposition \ref{Prop3} to the present situation, we can easily see that
\[\sup_{t \in\,\reals}\E\le|\eta^{x_1}_n(t)-\eta^{x_2}_n(t)\r|_E^2\leq c_n\,|x_1-x_2|_E^2,\]
and this allows us to conclude that
\[\sup_{t \in\,\reals}\le|\int_E B_{1,n}(x_1,y)\,\mu_t^{n,x_1}(dy)-\int_E B_{1,n}(x_2,y)\,\mu_t^{n,x_2}(dy)\r|_E\leq c_n\,|x_1-x_2|_E,\]
which implies the Lipschitz continuity of $\bar{B}_{n}$.
\end{proof}

\bigskip

As in \cite{av2} and \cite{pol}, we prove the validity of Lemma \ref{new62} by using  Khasminskii's approach based on time discretization, as introduced in \cite{khas}.

To this purpose, for any $\e>0$ we divide the interval $[0,T]$ in subintervals of size $\d_\e>0$, for some constant $\d_\e>0$ to be  determined, and we introduce the auxiliary fast motion $\hat{v}_{\e,n}$ defined in each time interval $[k\d_\e, (k+1)\d_\e]$, for $k=0,1,\ldots,[T/\d_\e]$, as the solution of the problem
\begin{equation}
\label{new40}
\left\{
\begin{array}{l}
\ds{
dv(t)=\frac 1\e\le[(A_2(t/\e)-\a) v(t)+B_{2,n}(t/\e,u_{\e,n}(k\d_\e),v(t))\r]\,dt+
\frac 1{\sqrt{\e}}\,G_{2}(t/\e,v(t))\,dw^{Q_2}(t),}\\
\vs
\ds{v(k\d_\e)=v_{\e,n}(k\d_\e).}
\end{array}\r.
\end{equation}
Notice that, due to the way $\hat{v}_{\e,n}$ has been defined, we have that an estimate analogous to \eqref{new2} holds, that is for any $p\geq 1$
\begin{equation}
\label{new43}
\int_0^T\E\,|\hat{v}_{\e,n}(t)|_E^p\,dt\leq c_{p,T}\le(1+|x|_E^p+|y|_E^p\r).
\end{equation}

As in \cite{khas} and \cite{av2}, we want to prove the following approximation result.

\begin{Lemma}
\label{lemma63}
Assume Hypotheses \ref{H0} to \ref{H3} and fix $x \in\,C^\theta(\bar{D})$ and $y \in\,E$. Then, there exists a constant $\kappa>0$ such that if
\[\d_\e=\e\,\log \e^{-\kappa},\]
then for any fixed $n \in\,\nat$ 
\begin{equation}
\label{new41}
\lim_{\e\to 0}\sup_{t \in\,[0,T]}\E|\hat{v}_{\e,n}(t)-v_{\e,n}(t)|_H^2=0.
\end{equation}
\end{Lemma}

\begin{proof}
Let $\e>0$ and $n \in\,\nat$ be fixed. For $k=0,\ldots,[T/\d_\e]$ and $t \in\,[k\d_\e,(k+1)\d_\e]$, let $\La_{\e,n}(t)$ be the solution of the problem
\[d\La_{\e,n}(t)=\frac 1\e\,(A_2-\a)\La_{\e,n}(t)\,dt+\frac 1{\sqrt{\e}}\,K_{\e,n}(t)\,dw^{Q_2},\ \ \ \ \La_{\e,n}(k\d_\e)=0,\]
where
\[K_{\e,n}(t):=G_{2,n}(t/\e,\hat{v}_{\e,n}(t))-G_{2,n}(t/\e,{v}_{\e,n}(t)).\]
Notice that, with the notations of Section \ref{sec2}, we can write
\begin{equation}
\label{fin1007}
\La_{\e,n}(t)=\psi_{\a,\e}(\La_{\e,n};k\d_\e)(t)+\Gamma_{\e,n}(t),\ \ \ \ \ t \in\,[k\d_\e,(k+1)\d_\e],
\end{equation}
where
\[\Gamma_{\e,n}(t)=\frac 1{\sqrt{\e}}\int_{k\d_\e}^t U_{\a,\e}(t,r)K_{\e,n}(r)\,dw^{Q_2}(r).\]
If  we define
$\rho_{\e,n}(t):=\hat{v}_{\e,n}(t)-v_{\e,n}(t)$ and 
$z_{\e,n}(t):=\rho_{\e,n}(t)-\La_{\e,n}(t)$, we have
\[d z_{\e,n}(t)=\frac 1\e\le[(A_2(t/\e)-\a) z_{\e,n}(t)+H_{\e,n}(t)\r]\,dt,\ \ \ \ \ z_{\e,n}(k\d_\e)=0,\]
where, in view of \eqref{nl102}
\[\begin{array}{l}
\ds{H_{\e,n}(t):=B_{2,n}(t/\e,u_{\e,n}(k\d_\e),\hat{v}_{\e,n}(t))-B_{2,n}(t/\e,u_{\e,n}(t),{v}_{\e,n}(t))}\\
\vs
\ds{=B_{2,n}(t/\e,u_{\e,n}(k\d_\e),\hat{v}_{\e,n}(t))-B_{2,n}(t/\e,u_{\e,n}(t),\hat{v}_{\e,n}(t))}\\
\vs
\ds{-\la(t/\e,\cdot,u_{\e,n}(t),\hat{v}_{\e,n}(t),{v}_{\e,n}(t))(z_{\e,n}(t)+\La_{\e,n}(t)). }
\end{array}\]
By proceeding as in \cite[proof of Lemma 6.2]{pol}, we have 
\[\begin{array}{l}
\ds{|z_{\e,n}(t)|_E\leq \frac{c_n}\e\int_{k\d_\e}^t e^{-\frac \a \e (t-s)}|u_{\e,n}(k\d_\e)-u_{\e,n}(s)|_E\,ds}\\
\vs
\ds{+\frac 1\e \int_{k\d_\e}^t\exp\le(-\frac 1\e\int_s^t \la_{\e,n}(r)\,dr\r)\la_{\e,n}(s)\,|\La_{\e,n}(s)|_E\,ds.}
\end{array}\]
where
\[\la_{\e,n}(t):=\la(t/\e,\xi_{\e,n}(t),u_{\e,n}(t,\xi_{\e,n}(t)),\hat{v}_{\e,n}(t,\xi_{\e,n}(t)),{v}_{\e,n}(t,\xi_{\e,n}(t))),\]
and $\xi_{\e,n}(t)$ is a point in $\bar{D}$ such that
\[|z_{\e,n}(t,\xi_{\e,n}(t))|=|z_{\e,n}(t)|_E.\]

Now, it is not difficult to check that an estimate analogous to \eqref{fin1000} is also valid for $u_{\e,n}$. Therefore, 
 we get
\begin{equation}
\label{new74}
\begin{array}{l}
\ds{\E\,|\hat{v}_{\e,n}(t)-v_{\e,n}(t)|_E^2 \leq c_p\,\E\,|\La_{\e,n}(t)|_E^2+c_{n}\le(  1+|x|_{C^\theta(\bar{D})}^{2 m_2}+|y|_E^2\r)\d_\e^{\gamma(\theta) 2}}\\
\vs
\ds{+c\,\E\sup_{s \in\,[k \d_\e,t]}|\La_{\e,n}(s)|_E^2\le(\frac 1\e \int_{k\d_\e}^t\exp\le(-\frac 1\e\int_s^t \la_{\e,n}(r)\,dr\r)\la_{\e,n}(s)\,ds\r)^2}\\
\vs
\ds{\leq c_{n}\le(  1+|x|_{C^\theta(\bar{D})}^{2 m_2}+|y|_E^2\r)\d_\e^{\gamma(\theta) 2}+c\,\E\sup_{s \in\,[k \d_\e,t]}|\La_{\e,n}(s)|_E^2.}
\end{array}
\end{equation}

Since for any $\a\geq 0$ and $\e>0$ we have
\[U_{\a,\e}(t,s)=U_{\a,\e}(t,r) U_{\a,\e}(r,s),\ \ \ \ s<r<t,\]
the usual factorization argument used in the autonomous case can be used also here, so that
 for $s \in\,[k\d_\e,(k+1)\d_\e]$ and $\eta \in\,(0,1)$ we have
\[\Gamma_{\e,n}(s)=\frac{\sin \pi \eta}{\pi}\frac 1{\sqrt{\e}}\int_{k\d_\e}^s(s-r)^{\eta-1}U_{\a,\e}(s,r)\,Y_{\eta,\e,n}(r)\,dr,\]
where
\[Y_{\eta,\e,n}(r)=\int_{k\d_\e}^r(r-\rho)^{-\eta}U_{\a,\e}(r,\rho)\,K_{\e,n}(\rho)\,dw^{Q_2}(\rho).\]

Therefore, by proceeding as in the proof of \cite[Lemma 6.2]{pol}, we have
\begin{equation}
\label{fin1006}
\begin{array}{l}
\ds{\E\sup_{s \in\,[k\d_\e,t]}|\Gamma_{\e,n}(s)|_E^2\leq + c_{\eta}\frac 1\e
\int_{k\d_\e}^t\E\,|\hat{v}_{\e,n}(s)-v_{\e,n}(s)|_E^2\,ds.
}
\end{array}\end{equation}
Thanks to \eqref{fin1007} and \eqref{fin1006}, this implies 
\[\E\sup_{s \in\,[k\d_\e,t]}|\La_{\e,n}(s)|_E^2\leq + c_{\eta}\frac 1\e
\int_{k\d_\e}^t\E\,|\hat{v}_{\e,n}(s)-v_{\e,n}(s)|_E^2\,ds,\]
so that, thanks to \eqref{new74}, for $t \in\,[k\d_\e,(k+1)\d_\e]$
\[\begin{array}{l}
\ds{\E\,|\hat{v}_{\e,n}(t)-v_{\e,n}(t)|_E^2\leq c_{\eta}\le(1+|x|_{C^\theta(\bar{D})}^{2 m_2}+|y|_E^2\r)\,\d_\e^{\gamma(\theta)2}+\frac {c}\e
\int_{k\d_\e}^t\E\,|\hat{v}_{\e,n}(s)-v_{\e,n}(s)|_E^2\,ds.}
\end{array}\]
From the Gronwall Lemma, this gives
\[\E\,|\hat{v}_{\e,n}(t)-v_{\e,n}(t)|_E^2\leq c_{\eta}\le(1+|x|_{C^\theta(\bar{D})}^{2 m_2}+|y|_E^2\r)\,\d_\e^{\gamma(\theta)2}\exp\le(\frac{c\,\d_\e}\e\r).\]
Now, since 
\[\exp\le(\frac{c\,\d_\e}\e\r)=\exp\le(c\,\log \e^{-\kappa}\r)=\e^{-c\,\kappa},\]
we have
\[\d_\e^{\gamma(\theta)2}\exp\le(\frac{c\,\d_\e}\e\r)=\d_\e^{\gamma(\theta)2}\e^{-c\,\kappa}=\e^{-c\,\kappa+2\gamma(\theta)}\le(\log \e^{-\kappa}\r)^{2\gamma(\theta)}.\]
Hence, if we take $\kappa<2\,\gamma(\theta)/c$, we have \eqref{new41}.

\end{proof}

Finally, we can prove \eqref{re}. As in \cite{pol}, we can show that
for any $n \in\,\nat$
\[\E\,\sup_{t \in\,[0,T]}|R_\e(t)|\leq \E\le(\sup_{t \in\,[0,T]}|R_{\e,n}(t)|\r)+\frac{c_T}n \le(1+|x|_E^{2 m_1}+|y|_E^2\r)|h|_E.\]
Therefore, due to the arbitrariness of $n \in\,\nat$, \eqref{re} follows once we prove that for any fixed $n \in\,\nat$
\begin{equation}
\label{new81}
\lim_{\e\to 0}\E\le(\sup_{t \in\,[0,T]}|R_{\e,n}(t)|\r)=0.
\end{equation}

We have 
\begin{equation}
\label{new84}\begin{array}{l}
\ds{\limsup_{\e\to 0}\,\E\sup_{t \in\,[0,T]}\le|\int_0^t\le<B_{1,n}(u_{\e,n}(s),v_{\e,n}(s))-\bar{B}_n(u_{\e,n}(s)),h\r>_H\,ds\r|}\\
\vs
\ds{\leq \limsup_{\e\to 0}\,\E\int_0^T\le|\le<B_{1,n}(u_{\e,n}(s),v_{\e,n}(s))-B_{1,n}(u_{\e,n}([s/\d_\e]\d_\e),\hat{v}_{\e,n}(s)),h\r>_H\r|\,ds}\\
\vs
\ds{+\limsup_{\e\to 0}\,\E\sup_{t \in\,[0,T]}\le|\int_0^t\le<B_{1,n}(u_{\e,n}([s/\d_\e]\d_\e),\hat{v}_{\e,n}(s))-\bar{B}_n(u_{\e,n}(s)),h\r>_H\,ds\r|.}
\end{array}\end{equation}

As in \cite[proof of Lemma 6.3]{av2} and \cite[proof of Lemma 6.2]{pol}, we have 
\[\begin{array}{l}
\ds{\E\int_0^T\le|\le<B_{1,n}(u_{\e,n}(s),v_{\e,n}(s))-B_{1,n}(u_{\e,n}([s/\d_\e]\d_\e),\hat{v}_{\e,n}(s)),h\r>_H\r|\,ds}\\
\vs
\ds{\leq c_{T,n}\,|h|_H \le(1+|x|_{C^\theta(\bar{D})}^{(2\vee \theta) m_1}+|y|_E^{2\vee \theta}\r)\le(\d_\e^{\gamma(\theta)}+\sup_{t \in\,[0,T]}\le(\E\,\le|v_{\e,n}(t)-\hat{v}_{\e,n}(t)\r|^2_E\r)^{\frac 12}\r).}
\end{array}\]
Therefore, in view of Lemma \ref{lemma63}, from \eqref{new84}
\[\begin{array}{l}
\ds{\limsup_{\e\to 0}\,\E\sup_{t \in\,[0,T]}\le|\int_0^t\le<B_{1,n}(u_{\e,n}(s),v_{\e,n}(s))-\bar{B}_n(u_{\e,n}(s)),h\r>_H\,ds\r|}\\
\vs
\ds{=\limsup_{\e\to 0}\,\E\sup_{t \in\,[0,T]}\le|\int_0^t\le<B_{1,n}(u_{\e,n}([s/\d_\e]\d_\e),\hat{v}_{\e,n}(s))-\bar{B}_n(u_{\e,n}(s)),h\r>_H\,ds\r|.}
\end{array}\]

Again, as in \cite[proof of Lemma 6.3]{av2} and \cite[proof of Lemma 6.2]{pol}, we have 
 \[\begin{array}{l}
\ds{\E\sup_{t \in\,[0,T]}\le|\int_0^t\le<B_{1,n}(u_{\e,n}([s/\d_\e]\d_\e),\hat{v}_{\e,n}(s))-\bar{B}_n(u_{\e,n}(s)),h\r>_H\,ds\r|}\\
\vs
\ds{\leq \sum_{k=0}^{[T/\d_\e]}\E\,\le|\int_{k\d_\e}^{(k+1)\d_\e}\le<B_{1,n}(u_{\e,n}([s/\d_\e]\d_\e),\hat{v}_{\e,n}(s))-\bar{B}_n(u_{\e,n}(k\d_\e)),h\r>_H\,ds\r|}\\
\vs
\ds{+ c_{T,n}\,|h|_H\le(1+|x|_{C^\theta(\bar{D})}^{m_1}+|y|_E\r)[T/\d_\e]\,\d_\e^{\gamma(\theta)+1},}
\end{array}\]
so that we have to show that

\begin{equation}
\label{new85}
\lim_{\e\to 0} \sum_{k=0}^{[T/\d_\e]}\E\,\le|\int_{k\d_\e}^{(k+1)\d_\e}\le<B_{1,n}(u_{\e,n}([s/\d_\e]\d_\e),\hat{v}_{\e,n}(s))-\bar{B}_n(u_{\e,n}(k\d_\e)),h\r>_H\,ds\r|=0.
\end{equation}
If we set $\zeta_\e:=\d_\e/\e$, we have
\[\begin{array}{l}
\ds{\E\,\le|\int_{k\d_\e}^{(k+1)\d_\e}\le<B_{1,n}(u_{\e,n}([s/\d_\e]\d_\e),\hat{v}_{\e,n}(s))-\bar{B}_n(u_{\e,n}(k\d_\e)),h\r>_H\,ds\r|}\\
\vs
\ds{=\E\,\le|\int_{0}^{\d_\e}\le<B_{1,n}(u_{\e,n}([s/\d_\e]\d_\e),\hat{v}_{\e,n}(k\d_\e s))-\bar{B}_n(u_{\e,n}(k\d_\e)),h\r>_H\,ds\r|
}\\
\vs
\ds{=\E\,\le|\int_{0}^{\d_\e}\le<B_{1,n}(u_{\e,n}([s/\d_\e]\d_\e),\tilde{v}_n^{u_{\e,n}(k\d_\e),v_{\e,n}(k\d_\e)}(s/\e)-\bar{B}_n(u_{\e,n}(k\d_\e)),h\r>_H\,ds\r|}\\
\vs
\ds{=\d_\e\E\,\le|\frac 1{\zeta_\e}\int_{0}^{\zeta_\e}\le<B_{1,n}(u_{\e,n}([s/\d_\e]\d_\e),\tilde{v}_n^{u_{\e,n}(k\d_\e),v_{\e,n}(k\d_\e)}(s)-\bar{B}_n(u_{\e,n}(k\d_\e)),h\r>_H\,ds\r|,}
\end{array}\]
where $\tilde{v}_n^{u_{\e,n}(k\d_\e),v_{\e,n}(k\d_\e)}(s)$ is the solution of the fast motion equation \eqref{fast}, with frozen slow component $u_{\e,n}(k\d_\e)$ and initial datum $v_{\e,n}(k\d_\e)$ and noise $\tilde{w}^{Q_2}$ independent of both of them. According to \eqref{new68}, \eqref{stima2} and \eqref{new2}, this yields
\[\begin{array}{l}
\ds{\E\,\le|\int_{k\d_\e}^{(k+1)\d_\e}\le<B_{1,n}(u_{\e,n}([s/\d_\e]\d_\e),\hat{v}_{\e,n}(s))-\bar{B}_n(u_{\e,n}(k\d_\e)),h\r>_H\,ds\r|}\\
\vs
\ds{\leq \d_\e \frac c{\zeta_\e}\le(1+\E\,|u_{\e,n}(k\d_\e)|^{\kappa_1}_E+\E\,|v_{\e,n}(k\d_\e)|^{\kappa_2}_E\r)|h|_1+\E\,\a(\zeta_\e,u_{\e,n}(k\d_\e)).}
\end{array}\]
Now, the family 
\[\le\{u_{\e,n}(k\d_\e)\,:\,\e>0,\ n \in\,\nat,\ k=0,\ldots,[T/\d_\e]\r\},\]
is tight. Then for any $\eta>0$ there exists a compact set $K_\eta\subset E$ such that
\[\Pro\le(u_{\e,n}(k\d_\e) \in\,K_\eta^c\r)\leq \eta.\]
Therefore, due to \eqref{fin1011} we have
\[\begin{array}{l}
\ds{\E\a(\zeta_\e,u_{\e,n}(k\d_\e))}\\
\vs
\ds{=\E\le(\a(\zeta_\e,u_{\e,n}(k\d_\e))\,;\,u_{\e,n}(k\d_\e) \in\,K_\eta\r)+\E\le(c\le(1+|u_{\e,n}(k\d_\e)|_E^{m_1}\r)\,;\,u_{\e,n}(k\d_\e) \in\,K^c_\eta\r)}\\
\vs
\ds{\leq \sup_{x \in\,K_\eta}\a(\zeta_\e,x)+\sqrt{\eta}\,c\le(1+\le(\E\,|u_{\e,n}(k\d_\e)|_E^{2m_1}\r)^{\frac 12}\r).}
\end{array}\]
Thanks to  \eqref{fin1010}, we can conclude that
\[\limsup_{\e\to 0}\E\,\a(\zeta_\e,u_{\e,n}(k\d_\e))\leq c\le(1+|x|_E^{\kappa}+|y|_E^\kappa\r)\sqrt{\eta},\]
for some $\kappa>0$, and, due to the arbitrariness of $\eta$ this implies 
 \eqref{new85}.

\end{document}